\documentclass[12pt]{amsart}

\usepackage[matrix,arrow,curve,frame]{xy}
\usepackage{amsmath,amsthm,amssymb,enumerate}
\usepackage{latexsym}
\usepackage{amscd}
\usepackage[colorlinks=false]{hyperref}
\usepackage{euscript}

\newtheorem{thm}[subsection]{Theorem}

\newtheorem{cor}[subsection]{Corollary}
\newtheorem{lemma}[subsection]{Lemma}

\theoremstyle{definition}

\numberwithin{equation}{section}

\def\cP{{\cal P}}
\def\cA{{\cal A}}

\def\ra{\rightarrow}
\def\bra{\langle}
\def\ket{\rangle}

\def\cA{{\mathcal A}}
\def\cB{{\mathcal B}}

\def\cH{{\mathcal H}}
\def\cI{{\mathcal I}}
\def\cJ{{\mathcal J}}

\def\cM{{\mathcal M}}

\def\cP{{\mathcal P}}

\def\cR{{\mathcal R}}
\def\cS{{\mathcal S}}

\def\cV{{\mathcal V}}
\def\cW{{\mathcal W}}

\def\gg{{\mathfrak g}}

\def\gl{{\mathfrak l}}

\def\go{{\mathfrak o}}

\def\gs{{\mathfrak s}}

\def\gB{{\mathfrak B}}

\newfont{\german}{eufm10}

\begin{document}
\pagestyle{plain}

\title
{Invariant subalgebras of affine vertex algebras}

\dedicatory{Dedicated to my father Michael A. Linshaw, M. D., on the occasion of his 70th birthday}

\author{Andrew R. Linshaw}
\address{Department of Mathemtics, Brandeis University}
\email{linshaw@brandeis.edu}

{\abstract
\noindent Given a finite-dimensional complex Lie algebra $\gg$ equipped with a nondegenerate, symmetric, invariant bilinear form $B$, let $V_k(\gg,B)$ denote the universal affine vertex algebra associated to $\gg$ and $B$ at level $k$. For any reductive group $G$ of automorphisms of $V_k(\gg,B)$, we show that the invariant subalgebra $V_k(\gg,B)^G$ is strongly finitely generated for generic values of $k$. This implies the existence of a new family of deformable $\cW$-algebras $\cW(\gg,B,G)_k$ which exist for all but finitely many values of $k$.}

\keywords{invariant theory; affine vertex algebra; current algebra; reductive group action; orbifold construction; strong finite generation; $\cW$-algebra}
\maketitle
\section{Introduction}

We call a vertex algebra $\cV$ {\it strongly finitely generated} if there exists a finite set of generators such that the collection of iterated Wick products of the generators and their derivatives spans $\cV$. Many known vertex algebras have this property, including affine, free field and lattice vertex algebras, as well as the $\cW$-algebras $\cW(\gg,f)_k$ associated via quantum Drinfeld-Sokolov reduction to a simple, finite-dimensional Lie algebra $\gg$ and a nilpotent element $f\in \gg$. Strong finite generation has many important consequences, and in particular implies that both Zhu's associative algebra $A(\cV)$, and Zhu's commutative algebra $\cV / C_2(\cV)$, are finitely generated.

In recent work, we have investigated the strong finite generation of invariant vertex algebras $\cV^G$, where $G$ is a reductive group of automorphisms of $\cV$. This is a vertex algebra analogue of Hilbert's theorem on the finite generation of classical invariant rings. It is a subtle and essentially \lq\lq quantum" phenomenon that is generally destroyed by passing to the classical limit before taking invariants. Often, $\cV$ admits a $G$-invariant filtration for which $\text{gr}(\cV)$ is a commutative algebra with a derivation (i.e., an abelian vertex algebra), and the classical limit $\text{gr}(\cV^G)$ is isomorphic to $(\text{gr}(\cV))^G$ as a commutative algebra. Unlike $\cV^G$, $\text{gr}(\cV^G)$ is generally not finitely generated as a vertex algebra, and a presentation will require both infinitely many generators and infinitely many relations.

Isolated examples of this phenomenon have been known for many years (see for example \cite{BFH}\cite{EFH}\cite{DN}\cite{FKRW}\cite{KWY}), although the first general results of this kind were obtained in \cite{LII}, in the case where $\cV$ is the $\beta\gamma$-system $\cS(V)$ associated to the vector space $V = \mathbb{C}^n$. The full automorphism group of $\cS(V)$ preserving a natural conformal structure is $GL_n$. By a theorem of Kac-Radul \cite{KR}, $\cS(V)^{GL_n}$ is isomorphic to the vertex algebra $\cW_{1+\infty}$ with central charge $-n$. In \cite{LI} we showed that $\cW_{1+\infty,-n}$ has a minimal strong generating set consisting of $n^2+2n$ elements, and in particular is a $\cW$-algebra of type $\cW(1,2,\dots, n^2+2n)$. For an arbitrary reductive group $G\subset GL_n$, $\cS(V)^G$ decomposes as a direct sum of irreducible, highest-weight $\cW_{1+\infty,-n}$-modules. The strong finite generation of $\cW_{1+\infty,-n}$ implies a certain finiteness property of the modules appearing in $\cS(V)^G$. This property, together with a classical theorem of Weyl, yields the strong finite generation of $\cS(V)^G$. Using the same approach, we also proved in \cite{LII} that invariant subalgebras of $bc$-systems and $bc \beta\gamma$-systems are strongly finitely generated.

In \cite{LIII} we initiated a similar study of the invariant subalgebras of the rank $n$ Heisenberg vertex algebra $\cH(n)$. The full automorphism group of $\cH(n)$ preserving a natural conformal structure is the orthogonal group $O(n)$. Motivated by classical invariant theory, we conjectured that $\cH(n)^{O(n)}$ is a $\cW$-algebra of type $\cW(2,4,\dots,n^2+3n)$. For $n=1$, this was already known to Dong-Nagatomo \cite{DN}, and we proved it for $n=2$ and $n=3$. We also showed that this conjecture implies the strong finite generation of $\cH(n)^G$ for an arbitrary reductive group $G$; see Theorem 6.9 of \cite{LIII}. 

In this paper, we study invariant subalgebras of the universal affine vertex algebra $V_k(\gg,B)$ for a finite-dimensional Lie algebra $\gg$ equipped with a nondegenerate, symmetric, invariant bilinear form $B$. In the special case where $\gg$ is simple and $B$ is the normalized Killing form, it is customary to denote $V_k(\gg,B)$ by $V_k(\gg)$. Recall that $V_k(\gg,B)$ has generators $X^{\xi}$, which are linear in $\xi\in\gg$, and satisfy the OPE relations $$X^{\xi}(z) X^{\eta}(w) \sim k B(\xi, \eta)(z-w)^{-2} + X^{[\xi,\eta]}(w)(z-w)^{-1}.$$ Let $G$ be a reductive group of automorphisms of $V_k(\gg,B)$ for all $k\in \mathbb{C}$. Our main result is the following.

\begin{thm} \label{mainthm} For any $\gg$, $B$, and $G$, $V_k(\gg,B)^G$ is strongly finitely generated for generic values of $k$, i.e., for $k\in \mathbb{C}\setminus K$ where $K$ is at most countable. \end{thm}

Note that when $\gg$ is abelian and $k\neq 0$, $V_k(\gg,B)^G \cong \cH(n)^G$ for $n=\text{dim}(\gg)$, so this result both improves and generalizes our earlier study of the vertex algebras $\cH(n)^G$. The proof of Theorem \ref{mainthm} is divided into three steps. The first step is to prove it in the special case where $\gg$ is abelian and $G= O(n)$. We will show that $\cH(n)^{O(n)}$ is of type $\cW(2,4,\dots, n^2+3n)$ for $n\leq 6$, and although we do not prove this conjecture in general, we will establish the strong finite generation of $\cH(n)^{O(n)}$ for all $n$. The second step (which is a minor modification of Theorem 6.9 of \cite{LIII}) is to show that the strong finite generation of $\cH(n)^{O(n)}$ implies the strong finite generation of $\cH(n)^G$ for an arbitrary reductive group $G$. The third step is to reduce the general case to the case where $\gg$ is abelian using a deformation argument.

Introduce a formal variable $\kappa$ satisfying $\kappa^2 = k$, and let $F$ denote the $\mathbb{C}$-algebra of rational functions in $\kappa$ of the form $\frac{p(\kappa)}{\kappa^d}$ where $\text{deg}(p) \leq d$. Let $\cV$ be the vertex algebra with coefficients in $F$ which is freely generated by $\{a^{\xi}|\ \xi\in\gg\}$, which satisfy the OPE relations
$$a^{\xi}(z) a^{\eta}(w) \sim B(\xi, \eta)(z-w)^{-2} + \frac{1}{\kappa} a^{[\xi,\eta]}(w)(z-w)^{-1}.$$ For each $k\neq 0$, we have an isomorphism $\cV / (\kappa - \sqrt{k}) \ra V_k(\gg,B)$ sending $a^{\xi_i} \mapsto \frac{1}{\sqrt{k}} X^{\xi_i}$, where $(\kappa - \sqrt{k})$ denotes the ideal generated by $\kappa - \sqrt{k}$. Moreover, the limit $\cV_{\infty} =\lim_{\kappa \ra \infty} \cV$ is a well-defined vertex algebra, and $\cV_{\infty} \cong \cH(n)$ since $B$ is nondegenerate.

Since the group $G$ acts on $V_k(\gg,B)$ for all $k$, $G$ also acts on $\cV$ and preserves the formal variable $\kappa$. We will see that $\cH(n)^G$ is a deformation of $V_k(\gg,B)^G$ in the sense that $$\lim_{\kappa \ra \infty} \cV^G \cong \cH(n)^G,\ \ \ \ \ \ \ \ V_k(\gg,B)^G \cong \cV^G / (\kappa - \sqrt{k}),\ \ \ \ \ \ \ \ k\neq 0.$$

Both $\cV^G$ and $\cH(n)^G$ possess good increasing filtrations, and we have isomorphisms of differential graded algebras $$\text{gr}(\cV^G) \cong F\otimes_{\mathbb{C}} (\text{Sym} \bigoplus_{j\geq 0} V_j)^G \cong F\otimes_{\mathbb{C}} \text{gr}(\cH(n)^G),\ \ \ \ \ \ \ \ V_j\cong \gg.$$
Here $R=(\text{Sym} \bigoplus_{j\geq 0} V_j)^G$ has a derivation $\partial$ defined by $\partial x_j = x_{j+1}$ for $x_j\in V_j$, and is naturally graded by degree and weight. Choose a generating set $S = \{s_i|\ i\in I\}$ for $R$ as a differential algebra. We may assume that $S$ consists of homogeneous elements and contains only finitely many elements in each weight. This set corresponds to strong generating sets $T = \{t_i|\ i\in I\}$ for $\cV^G$, and $U = \{u_i|\ i\in I\}$ for $\cH(n)^G$. Since $\cH(n)^{G}$ is strongly finitely generated, we may assume without loss of generality that there is a finite subset $U' = \{u_1,\dots, u_s\}\subset U$, which strongly generates $\cH(n)^G$. This implies that each $u\in U \setminus U'$ admits a {\it decoupling relation} \begin{equation} \label{decouphintro} u = P(u_1,\dots, u_s),\end{equation} where $P$ is a normally ordered polynomial in $u_1,\dots, u_s$ and their derivatives.

Given a subset $K\subset \mathbb{C}$ containing $0$, let $F_K$ be the $\mathbb{C}$-algebra of rational functions in $\kappa$ of the form $\frac{p(\kappa)}{q(\kappa)}$ where $\text{deg}(p) \leq \text{deg}(q)$, and all roots of $q$ lie in $K$. For all $k$ such that $\sqrt{k} \notin K$, we have \begin{equation} \label{specialk} (F_K \otimes_F \cV^G)/( \kappa - \sqrt{k}) \cong \cV^G /( \kappa - \sqrt{k})\cong V_k(\gg,B)^G.\end{equation} 
We will find a set $K$ which is at most countable, such that the decoupling relations (\ref{decouphintro}) can be used to construct analogous decoupling relations \begin{equation} \label{decoupvkintro} t = Q(t_1,\dots, t_s),\end{equation} in $F_K \otimes_F \cV^G$ for all $t\in T\setminus T'$. Here $Q$ is a normally ordered polynomial in $t_1,\dots, t_s$ and their derivatives, with coefficients in $F_K$. This implies that $T'$ is a strong generating set for $F_K \otimes_F \cV^G$. Finally, for all $k$ such that $\sqrt{k} \notin K$, the image of $T'$ under \eqref{specialk} strongly generates $V_k(\gg,B)^G$.

The idea of studying $V_k(\gg,B)^G$ via the simpler object $\cH(n)^G$ is analogous to studying classical invariant rings of the form $U(\gg)^G$ via the Poincar\'e-Birkhoff-Witt abelianization $\text{Sym}(\gg)^G \cong \text{gr}(U(\gg)^G) \cong \text{gr}(U(\gg))^G$. Here $U(\gg)$ denotes the universal enveloping algebra of $\gg$, and $G$ is a group of automorphisms of $U(\gg)$. A generating set for $U(\gg)^G$ can be obtained from a generating set for $\text{Sym}(\gg)^G$, and the leading term of a relation among the generators of $U(\gg)^G$ corresponds to a relation among the generators of $\text{Sym}(\gg)^G$. However, more subtle information about the structure and representation theory of $U(\gg)^G$ cannot be reconstructed in this way. Likewise, we can view $\cH(n)$ as a kind of partial abelianization of $V_k(\gg,B)$. By passing to the full abelianization $\text{gr}(V_k(\gg,B)) \cong \text{Sym} \bigoplus_{j\geq 0} V_j$, we destroy too much structure, and $\text{gr}(V_k(\gg,B))^G$ generally fails to be finitely generated. By contrast, $\cH(n)$ retains enough structure (notably, it is still a simple vertex algebra) so that strong generating sets for $\cH(n)^G$ and $V_k(\gg,B)^G$ are closely related. 

A {\it deformable $\cW$-algebra} is a family of vertex algebras $\cW_k$ depending on a parameter $k$, equipped with strong generating sets $A_k = \{a^k_1,\dots,a^k_r\}$, whose structure constants are continuous functions of $k$ with isolated singularities. The structure constants are just the coefficients of each normally ordered monomial in the elements of $A_k$ and their derivatives appearing in the OPE of $a^k_i(z) a^k_j(w)$, for $i,j=1,\dots,r$. It is customary in the physics literature to assume that one of the generators is a Virasoro element and the remaining generators are primary, but in our definition we do not assume that the remaining generators are primary. An easy consequence of Theorem \ref{mainthm} is the following.
\begin{cor} \label{maincor} For any $\gg$, $B$, and $G$ as above, there is a deformable $\cW$-algebra $\cW(\gg,B,G)_k$ which exists for all but finitely many values of $k$. For generic values of $k$, $\cW(\gg,B,G)_k=V_k(\gg,B)^G$.
\end{cor}

Perhaps the best-studied examples of deformable $\cW$-algebras are the algebras $\cW(\gg,f)_k$ associated to a simple, finite-dimensional Lie algebra $\gg$, and a nilpotent element $f \in \gg$ (see \cite{KRW} and references therein). A key property of $\cW(\gg,f)_k$ is that it is {\it freely generated}; there exists a strong finite generating set such that there are no nontrivial normally ordered polynomial relations among the generators and their derivatives \cite{DSK}. By contrast, $\cW(\gg,B,G)_k$ is generally {\it not} freely generated. To see this, let $\cA$ be a vertex algebra with conformal weight grading $\cA = \bigoplus_{m\geq 0} \cA[m]$, which is freely generated by fields $a_1,\dots, a_r$ of weights $w_1,\dots, w_r$. Then $\cA$ is linearly isomorphic to the polynomial algebra with generators $a_1,\dots, a_r$ and their derivatives, where $\text{wt}(\partial^k a_i) = w_i +k$, so $\cA$ has graded character $$\chi_{\cA}(q) =\sum_{m\geq 0}\text{dim}(\cA[m]) q^m = \prod_{i=1}^r \prod_{k\geq 0} \frac{1}{1-q^{w_i +k}}.$$ On the other hand, $V_k(\gg,B)^G$ is linearly isomorphic to $(\text{Sym} \bigoplus_{j\geq 0} V_j)^G$, where $V_j \cong \gg$ and has weight $j+1$. This generally does not have the graded character of a polynomial algebra.

The vertex algebras $\cW(\gg,B,G)_k$ form a new and rich class of deformable $\cW$-algebras, and it is important to study their representation theory as well as their structure. For generic values of $k$, $\cW(\gg,B,G)_k$ will be simple, but for special values of $k$ these algebras will possess nontrivial ideals. It is likely that for positive integer values of $k$, there exist new rational vertex algebras arising as quotients of vertex algebras of this kind.

\section{Vertex algebras}
In this section, we define vertex algebras, which have been discussed from various different points of view in the literature (see for example \cite{Bo}\cite{FLM}\cite{K}\cite{FBZ}). We will follow the formalism developed in \cite{LZ} and partly in \cite{LiI}. Let $V=V_0\oplus V_1$ be a super vector space over $\mathbb{C}$, and let $z,w$ be formal variables. By $\text{QO}(V)$, we mean the space of all linear maps $$V\rightarrow V((z)):=\{\sum_{n\in\mathbb{Z}} v(n) z^{-n-1}|
v(n)\in V,\ v(n)=0\ \text{for} \ n>>0 \}.$$ Each element $a\in \text{QO}(V)$ can be
uniquely represented as a power series
$$a=a(z):=\sum_{n\in\mathbb{Z}}a(n)z^{-n-1}\in \text{End}(V)[[z,z^{-1}]].$$ We
refer to $a(n)$ as the $n$th Fourier mode of $a(z)$. Each $a\in
\text{QO}(V)$ is assumed to be of the shape $a=a_0+a_1$ where $a_i:V_j\ra V_{i+j}((z))$ for $i,j\in\mathbb{Z}/2\mathbb{Z}$, and we write $|a_i| = i$.

On $\text{QO}(V)$ there is a set of nonassociative bilinear operations
$\circ_n$, indexed by $n\in\mathbb{Z}$, which we call the $n$th circle
products. For homogeneous $a,b\in \text{QO}(V)$, they are defined by
$$
a(w)\circ_n b(w)=\text{Res}_z a(z)b(w)~\iota_{|z|>|w|}(z-w)^n-
(-1)^{|a||b|}\text{Res}_z b(w)a(z)~\iota_{|w|>|z|}(z-w)^n.
$$
Here $\iota_{|z|>|w|}f(z,w)\in\mathbb{C}[[z,z^{-1},w,w^{-1}]]$ denotes the
power series expansion of a rational function $f$ in the region
$|z|>|w|$. We usually omit the symbol $\iota_{|z|>|w|}$ and just
write $(z-w)^{-1}$ to mean the expansion in the region $|z|>|w|$,
and write $-(w-z)^{-1}$ to mean the expansion in $|w|>|z|$. It is
easy to check that $a(w)\circ_n b(w)$ above is a well-defined
element of $\text{QO}(V)$.

The non-negative circle products are connected through the {\it
operator product expansion} (OPE) formula.
For $a,b\in \text{QO}(V)$, we have \begin{equation} \label{opeformula} a(z)b(w)=\sum_{n\geq 0}a(w)\circ_n
b(w)~(z-w)^{-n-1}+:a(z)b(w):\ ,\end{equation} which is often written as
$a(z)b(w)\sim\sum_{n\geq 0}a(w)\circ_n b(w)~(z-w)^{-n-1}$, where
$\sim$ means equal modulo the term $$
:a(z)b(w):\ =a(z)_-b(w)\ +\ (-1)^{|a||b|} b(w)a(z)_+.$$ Here
$a(z)_-=\sum_{n<0}a(n)z^{-n-1}$ and $a(z)_+=\sum_{n\geq
0}a(n)z^{-n-1}$. Note that $:a(w)b(w):$ is a well-defined element of
$\text{QO}(V)$. It is called the {\it Wick product} of $a$ and $b$, and it
coincides with $a\circ_{-1}b$. The other negative circle products
are related to this by
$$ n!~a(z)\circ_{-n-1}b(z)=\ :(\partial^n a(z))b(z):\ ,$$
where $\partial$ denotes the formal differentiation operator
$\frac{d}{dz}$. For $a_1(z),...,a_k(z)\in \text{QO}(V)$, the $k$-fold
iterated Wick product is defined to be
\begin{equation} \label{iteratedwick} :a_1(z)a_2(z)\cdots a_k(z):\ =\ :a_1(z)b(z):~,\end{equation}
where $b(z)=\ :a_2(z)\cdots a_k(z):\ $. We often omit the formal variable $z$ when no confusion can arise.

The set $\text{QO}(V)$ is a nonassociative algebra with the operations
$\circ_n$ and a unit $1$. We have $1\circ_n a=\delta_{n,-1}a$ for
all $n$, and $a\circ_n 1=\delta_{n,-1}a$ for $n\geq -1$. A linear subspace $\cA\subset \text{QO}(V)$ containing 1 which is closed under the circle products will be called a {\it quantum operator algebra} (QOA).
In particular $\cA$ is closed under $\partial$
since $\partial a=a\circ_{-2}1$. Many formal algebraic
notions are immediately clear: a homomorphism is just a linear
map that sends $1$ to $1$ and preserves all circle products; a module over $\cA$ is a
vector space $M$ equipped with a homomorphism $\cA\rightarrow
\text{QO}(M)$, etc. A subset $S=\{a_i|\ i\in I\}$ of $\cA$ is said to {\it generate} $\cA$ if any element $a\in\cA$ can be written as a linear
combination of nonassociative words in the letters $a_i$, $\circ_n$, for
$i\in I$ and $n\in\mathbb{Z}$. We say that $S$ {\it strongly generates} $\cA$ if any $a\in\cA$ can be written as a linear combination of words in the letters $a_i$, $\circ_n$ for $n<0$. Equivalently, $\cA$ is spanned by the collection $\{ :\partial^{k_1} a_{i_1}(z)\cdots \partial^{k_m} a_{i_m}(z):| ~i_1,\dots,i_m \in I,~ k_1,\dots,k_m \geq 0\}$. We say that $S$ {\it freely generates} $\cA$ if there are no nontrivial normally ordered polynomial relations among the generators and their derivatives.

We say that $a,b\in \text{QO}(V)$ {\it quantum commute} if $(z-w)^N
[a(z),b(w)]=0$ for some $N\geq 0$. Here $[,]$ denotes the super bracket. This condition implies that $a\circ_n b = 0$ for $n\geq N$, so (\ref{opeformula}) becomes a finite sum. A commutative QOA (CQOA) is a QOA whose elements pairwise quantum commute. Finally, the notion of a CQOA is equivalent to the notion of a vertex algebra. Every CQOA $\cA$ is itself a faithful $\cA$-module, called the {\it left regular
module}. Define
$$\rho:\cA\rightarrow \text{QO}(\cA),\ \ \ \ a\mapsto\hat a,\ \ \ \ \hat
a(\zeta)b=\sum_{n\in\mathbb{Z}} (a\circ_n b)~\zeta^{-n-1}.$$ Then $\rho$ is an injective QOA homomorphism,
and the quadruple of structures $(\cA,\rho,1,\partial)$ is a vertex
algebra in the sense of \cite{FLM}. Conversely, if $(V,Y,{\bf 1},D)$ is
a vertex algebra, the collection $Y(V)\subset \text{QO}(V)$ is a
CQOA. {\it We will refer to a CQOA simply as a
vertex algebra throughout the rest of this paper}.

The main examples we consider are the {\it universal affine vertex algebras} and their invariant subalgebras. Let $\gg$ be a finite-dimensional Lie algebra over $\mathbb{C}$, equipped with a nondegenerate, symmetric, invariant bilinear form $B$. The loop algebra $\gg[t,t^{-1}] = \gg\otimes \mathbb{C}[t,t^{-1}]$ has a one-dimensional central extension $\hat{\gg} = \gg[t,t^{-1}]\oplus \mathbb{C}\kappa$ determined by $B$, with bracket $$[\xi t^n, \eta t^m] = [\xi,\eta] t^{n+m} + n B(\xi,\eta) \delta_{n+m,0} \kappa,$$ and $\mathbb{Z}$-gradation $\text{deg}(\xi t^n) = n$, $\text{deg}(\kappa) = 0$. Let $\hat{\gg}_{\geq 0} = \bigoplus_{n\geq 0} \hat{\gg}_n$ where $\hat{\gg}_n$ denotes the subspace of degree $n$. For $k\in \mathbb{C}$, let $C_k$ be the one-dimensional $\hat{\gg}_{\geq 0}$-module on which $\xi t^n$ acts trivially for $n\geq 0$, and $\kappa$ acts by the $k$ times the identity. Define $V_k = U(\hat{\gg})\otimes_{U(\hat{\gg}_{\geq 0})} C_k$, and let $X^{\xi}(n)\in \text{End}(V_k)$ be the linear operator representing $\xi t^n$ on $V_k$. Define $X^{\xi} (z) = \sum_{n\in\mathbb{Z}} X^{\xi} (n) z^{-n-1}$, which is easily seen to lie in $\text{QO}(V_k)$ and satisfy the OPE relation $$X^{\xi}(z)X^{\eta} (w)\sim kB(\xi,\eta) (z-w)^{-2} + X^{[\xi,\eta]}(w) (z-w)^{-1} .$$ The vertex algebra $V_k(\gg,B)$ generated by $\{X^{\xi}| \ \xi \in\gg\}$ is known as the universal affine vertex algebra associated to $\gg$ and $B$ at level $k$. 

Suppose that $\gg$ is simple and $B$ is the normalized Killing form $\bra,\ket = \frac{1}{2 h^{\vee}} \bra ,\ket_K$. In this case, it is customary to denote $V_k(\gg,B)$ by $V_k(\gg)$. Fix an orthonormal basis $\{\xi_1,\dots,\xi_n\}$ for $\gg$ relative to $\bra,\ket$. For $k\neq - h^{\vee}$ where $h^{\vee}$ is the dual Coxeter number, $V_k(\gg)$ is a conformal vertex algebra with Virasoro element $$L(z) = \frac{1}{2(k+h^{\vee})} \sum_{i=1}^n :X^{\xi_i}(z) X^{\xi_i}(z):,$$ of central charge $\frac{k\cdot \text{dim}(\gg)}{k+h^{\vee}}$, such that each $X^{\xi_i}$ is primary of weight one. This Virasoro element is known as the {\it Sugawara conformal vector}, and it lies in $V_k(\gg)^G$ for any group $G$ of automorphisms of $V_k(\gg)$. At the critical level $k = - h^{\vee}$, $L(z)$ does not exist, but $V_k(\gg)$ still possesses a {\it quasi-conformal structure}; there is an action of the Lie subalgebra $\{L_n|~n\geq -1\}$ of the Virasoro algebra, such that $L_{-1}$ acts by translation, $L_0$ acts diagonalizably, and each $X^{\xi_i}$ has weight one. In fact, this quasiconformal structure exists on $V_k(\gg,B)$ for any $\gg$ and $B$, so we always have the weight grading $V_k(\gg,B) = \bigoplus_{m\geq 0} V_k(\gg,B)[m]$. Any group $G$ of automorphisms of $V_k(\gg,B)$ for all $k$ must act on $\gg$ and preserve both the bracket and the bilinear form, and we have the weight grading \begin{equation} \label{cwgr} V_k(\gg,B)^G = \bigoplus_{m\geq 0} V_k(\gg,B)^G[m].\end{equation}

Next, suppose that $\gg$ is abelian. Since $B$ is nondegenerate, $V_k(\gg,B)$ is just the rank $n$ Heisenberg vertex algebra $\cH(n)$. If we choose an orthonormal basis $\{\xi_1,\dots,\xi_n\}$ for $\gg$, $\cH(n)$ is generated by $\{\alpha^i = X^{\xi_i}|~i=1,\dots, n\}$, satisfying the OPE relations $$\alpha^i(z) \alpha^j(w) \sim \delta_{i,j} (z-w)^{-2}.$$ There is a Virasoro element $L(z) = \frac{1}{2} \sum_{i=1}^n :\alpha^i(z) \alpha^i(z):$ of central charge $n$, under which each $\alpha^i$ is primary of weight one.

\section{Category $\mathcal{R}$}
Let $\cR$ be the category of vertex algebras $\cA$ equipped with a $\mathbb{Z}_{\geq 0}$-filtration
\begin{equation} \cA_{(0)}\subset\cA_{(1)}\subset\cA_{(2)}\subset \cdots,\ \ \ \ \ \ \ \cA = \bigcup_{k\geq 0}
\cA_{(k)}\end{equation} such that $\cA_{(0)} = \mathbb{C}$, and for all
$a\in \cA_{(k)}$, $b\in\cA_{(l)}$, we have
\begin{equation} \label{goodi} a\circ_n b\in\cA_{(k+l)},\ \ \ \text{for}\
n<0,\end{equation}
\begin{equation} \label{goodii} a\circ_n b\in\cA_{(k+l-1)},\ \ \ \text{for}\
n\geq 0.\end{equation}
Elements $a(z)\in\cA_{(d)}\setminus \cA_{(d-1)}$ are said to have degree $d$.

Filtrations on vertex algebras satisfying (\ref{goodi})-(\ref{goodii})~were introduced in \cite{LiII}, and are known as {\it good increasing filtrations}. Setting $\cA_{(-1)} = \{0\}$, the associated graded object $\text{gr}(\cA) = \bigoplus_{k\geq 0}\cA_{(k)}/\cA_{(k-1)}$ is a $\mathbb{Z}_{\geq 0}$-graded associative, (super)commutative algebra with a
unit $1$ under a product induced by the Wick product on $\cA$. For each $r\geq 1$ we have the projection \begin{equation} \phi_r: \cA_{(r)} \ra \cA_{(r)}/\cA_{(r-1)}\subset \text{gr}(\cA).\end{equation} 
Moreover, $\text{gr}(\cA)$ has a derivation $\partial$ of degree zero
(induced by the operator $\partial = \frac{d}{dz}$ on $\cA$), and
for each $a\in\cA_{(d)}$ and $n\geq 0$, the operator $a\circ_n$ on $\cA$
induces a derivation of degree $d-k$ on $\text{gr}(\cA)$, which we denote by $a(n)$. Here $$k  = \text{sup} \{ j\geq 1|~ \cA_{(r)}\circ_n \cA_{(s)}\subset \cA_{(r+s-j)}~\forall r,s,n\geq 0\},$$ as in \cite{LL}. Finally, these derivations give $\text{gr}(\cA)$ the structure of a vertex Poisson algebra.

The assignment $\cA\mapsto \text{gr}(\cA)$ is a functor from $\cR$ to the category of $\mathbb{Z}_{\geq 0}$-graded (super)commutative rings with a differential $\partial$ of degree 0, which we will call $\partial$-rings. A $\partial$-ring is just an {\it abelian} vertex algebra, that is, a vertex algebra $\cV$ in which $[a(z),b(w)] = 0$ for all $a,b\in\cV$. A $\partial$-ring $A$ is said to be generated by a subset $\{a_i|~i\in I\}$ if $\{\partial^k a_i|~i\in I, k\geq 0\}$ generates $A$ as a graded ring. The key feature of $\cR$ is the following reconstruction property \cite{LL}:

\begin{lemma}\label{reconlem}Let $\cA$ be a vertex algebra in $\cR$ and let $\{a_i|~i\in I\}$ be a set of generators for $\text{gr}(\cA)$ as a $\partial$-ring, where $a_i$ is homogeneous of degree $d_i$. If $a_i(z)\in\cA_{(d_i)}$ are vertex operators such that $\phi_{d_i}(a_i(z)) = a_i$, then $\cA$ is strongly generated as a vertex algebra by $\{a_i(z)|~i\in I\}$.\end{lemma}

As shown in \cite{LI}, there is a similar reconstruction property for kernels of surjective morphisms in $\cR$. Let $f:\cA\rightarrow \cB$ be a morphism in $\cR$ with kernel $\cJ$, such that $f$ maps $\cA_{(k)}$ onto $\cB_{(k)}$ for all $k\geq 0$. The kernel $J$ of the induced map $\text{gr}(f): \text{gr}(\cA)\rightarrow \text{gr}(\cB)$ is a homogeneous $\partial$-ideal (i.e., $\partial J \subset J$). A set $\{a_i|~i\in I\}$ such that $a_i$ is homogeneous of degree $d_i$ is said to generate $J$ as a $\partial$-ideal if $\{\partial^k a_i|~i\in I,~k\geq 0\}$ generates $J$ as an ideal.

\begin{lemma} \label{idealrecon} Let $\{a_i| i\in I\}$ be a generating set for $J$ as a $\partial$-ideal, where $a_i$ is homogeneous of degree $d_i$. Then there exist vertex operators $a_i(z)\in \cA_{(d_i)}$ with $\phi_{d_i}(a_i(z)) = a_i$, such that $\{a_i(z)|~i\in I\}$ generates $\cJ$ as a vertex algebra ideal.\end{lemma}

For any Lie algebra $\gg$ and bilinear form $B$, $V_k(\gg,B)$ admits a good increasing filtration \begin{equation} \label{filto} V_k(\gg,B)_{(0)}\subset V_k(\gg,B)_{(1)}\subset \cdots,\ \ \ \ \ \ \ \ V_k(\gg,B) = \bigcup_{j\geq 0} V_k(\gg,B)_{(j)},\end{equation} where $V_k(\gg,B)_{(j)}$ is spanned by iterated Wick products of the generators $X^{\xi_i}$ and their derivatives, of length at most $j$. Since $V_k(\gg,B)$ is freely generated by $\{X^{\xi_i}\}$, we have a linear isomorphism $V_k(\gg,B) \cong \text{gr}(V_k(\gg,B))$, and an isomorphism of $\partial$-rings \begin{equation}\label{assgrad} \text{gr}(V_k(\gg,B)) \cong \text{Sym}\bigoplus_{j\geq 0} V_j,\ \ \ \ \ \ \ \ V_j\cong \gg. \end{equation} Here $V_j$ has a basis $\{\xi_{i,j}\}$ corresponding to $\{\partial^j X^{\xi_i}\}$, and $\text{Sym} \bigoplus_{j\geq 0} V_j$ has the $\partial$-ring structure defined by $\partial \xi_{i,j} = \xi_{i,j+1}$. For any reductive group $G$ of automorphisms of $V_k(\gg,B)$, this filtration is $G$-invariant and is inherited by $V_k(\gg,B)^G$. We obtain a linear isomorphism $V_k(\gg,B)^G \cong \text{gr}(V_k(\gg,B)^G)$ and isomorphisms of $\partial$-rings \begin{equation} \label{currentringiso}\text{gr}(V_k(\gg,B))^G \cong \text{gr}(V_k(\gg,B)^G) \cong (\text{Sym} \bigoplus_{j\geq 0} V_j)^G. \end{equation} Finally, the weight grading \eqref{cwgr} is inherited by $\text{gr}(V_k(\gg,B)^G)$ and \eqref{currentringiso} preserves weight as well as degree, where $\text{wt}(\xi_{i,j}) = j+1$.

\section{Some classical invariant theory \label{secivt}}

We recall some facts about invariant rings of the form $R = (\text{Sym} \bigoplus_{j\geq 0} V_j)^G$, where each $V_j$ is isomorphic to some fixed $G$-module $V$. In the terminology of Weyl, a {\it first fundamental theorem of invariant theory} for the pair $(G,V)$ is a set of generators for $R$. In some treatments, a first fundamental theorem is defined as a set of generators for the larger ring $(\text{Sym} \bigoplus_{j\geq 0} (V_j \oplus V^*_j))^G$, where $V^*_j\cong V^*$ as $G$-modules, but here we only need to consider $R$. A {\it second fundamental theorem of invariant theory} for $(G,V)$ is a set of generators for the ideal of relations among the generators of $R$. 

First and second fundamental theorems of invariant theory are known for the standard representations of the classical groups \cite{W} and for the adjoint representations of the classical groups \cite{P}, but in general it is quite difficult to describe these rings explicitly. However, there is an important finiteness theorem due to Weyl (Theorem 2.5A of \cite{W}). For all $p\geq 0$, there is an action of $GL_p$ on $\bigoplus_{j =0}^{p-1} V_j $ which commutes with the action of $G$. The inclusions $GL_p\hookrightarrow GL_q$ for $p<q$ sending $$M \rightarrow  \bigg[ \begin{matrix}M & 0 \cr 0 & I_{q-p} \end{matrix} \bigg]$$ induce an action of $GL_{\infty} = \lim_{p\rightarrow \infty} GL_p$ on $\bigoplus_{j\geq 0} V_j$. We obtain an action of $GL_{\infty}$ on $\text{Sym} \bigoplus_{j\geq 0} V_j$ which commutes with the action of $G$, so $GL_{\infty}$ acts on $R$ as well. The elements $\sigma \in GL_{\infty}$ are known as {\it polarization operators}, and given $f\in R$, $\sigma f$ is known as a polarization of $f$. 

\begin{thm} \label{weylfinite} (Weyl) $R$ is generated by the set of polarizations of any set of generators for $(\text{Sym} \bigoplus_{j = 0} ^{n-1} V_j)^G$, where $n=\text{dim}(V)$. Since $G$ is reductive, $(\text{Sym} \bigoplus_{j = 0} ^{n-1} V_j)^G$ is finitely generated, so there exists a finite subset $\{f_1,\dots, f_r\}\subset R$ whose polarizations generate $R$. \end{thm}

The ring $R$ is naturally graded by degree and weight, where $x_j \in V_j$ has degree one and weight $j+1$. The derivation $\partial$ on $R$ defined by $\partial x_j = x_{j+1}$ is therefore homogeneous of degree zero and weight one. For our purposes, we need a set $S = \{s_i|\ i\in I\}$ of generators for $R$ as a $\partial$-ring, so that $\{\partial^j s_i|\ j\geq 0,\ i\in I\}$ generates $R$ as a ring. We may assume that $S$ consists of polarizations of a finite set $\{f_1,\dots, f_r\}$, each $s_i$ is homogeneous with respect to degree and weight, and $S$ contains only finitely many elements in each weight.

Let $\cV$ be a vertex algebra with a weight grading $\cV = \bigoplus_{m\geq 0} \cV[m]$, an action of $G$ by automorphisms, and a $G$-invariant good increasing filtration such that $\text{gr}(\cV)^G \cong \text{gr}(\cV^G) \cong R$ as $\partial$-rings graded by both degree and weight. Choose a generating set $S = \{s_i|\ i\in I\}$ for $R$ as a $\partial$-ring, where $s_i$ has degree $d_i$ and weight $w_i$. By Lemma \ref{reconlem}, we can find a strong generating set $T = \{t_i|\ i\in I\}$ for $\cV^G$ such that $s_i = \phi_{d_i}(t_i)$, where $$\phi_{d_i}: (\cV^G)_{(d_i)} \rightarrow (\cV^G)_{(d_i)} /(\cV^{G})_{(d_i-1)} \subset \text{gr}(\cV^G)$$ is the usual projection.

Given a homogeneous polynomial $p\in R$ of degree $d$, a {\it normal ordering} of $p$ will be a choice of normally ordered polynomial $P\in (\cV^G)_{(d)}$, obtained by replacing each $s_i\in S$ by $t_i\in T$, and replacing ordinary products with iterated Wick products of the form (\ref{iteratedwick}). Of course $P$ is not unique, but for any choice of $P$ we have $\phi_{d}(P) = p$.

Suppose that $p$ is a relation among the generators of $R$ coming from the second fundamental theorem for $(G,V)$, which we may assume to be homogeneous of degree $d$. Let $P^d \in \cV^G$ be some normal ordering of $p$. Since $\text{gr}(\cV^G)\cong R$ as graded rings, it follows that $P^d$ lies in $(\cV^G)_{(d-1)}$. The polynomial $\phi_{d-1} (P^d) \in R$ is homogeneous of degree $d-1$; if it is nonzero, it can be expressed as a polynomial in the variables $s_i\in S$ and their derivatives. Choose some normal ordering of this polynomial, and call this vertex operator $-P^{d-1}$. Then $P^{d} + P^{d-1}$ has the property that $$ \phi_d (P^d + P^{d-1}) = p,\ \ \ \ \ P^d + P^{d-1} \in (\cV^G)_{(d-2)}.$$ Continuing this process, we arrive at a vertex operator \begin{equation} \label{quantumcorr} P=\sum_{k=1}^{d} P^{k} \in \cV^G\end{equation} which is identically zero. We view $P$ as a quantum correction of the relation $p$, and it is easy to see by induction on degree that all normally ordered polynomial relations in $\cV^G$ among the elements of $T$ and their derivatives, are consequences of relations of this kind.

In general, $R$ is not finitely generated as a $\partial$-ring, but since the relations (\ref{quantumcorr}) are more complicated than their classical counterparts, it is possible for $\cV^G$ to be strongly generated as a vertex algebra by a finite subset $T'\subset T$. For this to happen, each element $t \in T\setminus T'$ must admit a \lq\lq decoupling relation" expressing it as a normally ordered polynomial in the elements of $T'$ and their derivatives. Given a relation in $\cV^G$ of the form (\ref{quantumcorr}), suppose that some $t \in T \setminus T'$ appears in $P^k$ for some $k<d$, with nonzero coefficient. If the remaining terms in (\ref{quantumcorr}) only depend on the elements of $T'$ and their derivatives, we can solve for $t$ to obtain such a decoupling relation. The existence of a complete set of decoupling relations for all $t \in T \setminus T'$ is a subtle and nonclassical phenomenon, since it depends on precise properties of the quantum corrections appearing in relations of the form (\ref{quantumcorr}).

\section{The structure of $\cH(n)^{O(n)}$}
The full automorphism group of $\cH(n)$ is the orthogonal group $O(n)$, and for any reductive group $G\subset O(n)$, $\cH(n)^G$ is governed by $\cH(n)^{O(n)}$ in the sense that it decomposes as a direct sum of irreducible, highest-weight $\cH(n)^{O(n)}$-modules. In this section, we recall the notation and main results of \cite{LIII} on the structure of $\cH(n)^{O(n)}$. 

Recall that $\cH(n)\cong \text{gr}(\cH(n))$ as vector spaces, and $\text{gr}(\cH(n))\cong \text{Sym} \bigoplus_{j\geq 0} V_j$ as $\partial$-rings. Here $V_j$ is spanned by $\{x_{i,j} |~ i=1,\dots,n\}$ where $x_{i,j}$ corresponds to $\partial^j \alpha^i$, so $V_j\cong \mathbb{C}^n$ as $O(n)$-modules for all $j\geq 0$. We have a linear isomorphism $\cH(n)^{O(n)}\cong \text{gr}(\cH(n)^{O(n)})$, and isomorphisms of $\partial$-rings \begin{equation} \label{structureofgrsinv}\text{gr}(\cH(n)^{O(n)} )\cong (\text{gr}(\cH(n)))^{O(n)} \cong (\text{Sym} \bigoplus_{j\geq 0} V_j)^{O(n)}.\end{equation} The following classical result of Weyl (Theorems 2.9A and 2.17A of \cite{W}) describes the generators and relations of the ring $(\text{Sym} \bigoplus_{j\geq 0} V_j)^{O(n)}$.

\begin{thm} \label{weylfft} For $j\geq 0$, let $V_j$ be the copy of the standard $O(n)$-module $\mathbb{C}^n$ with orthonormal basis $\{x_{i,j}| ~i=1,\dots,n\}$. The ring $R=(\text{Sym} \bigoplus_{j\geq 0} V_j)^{O(n)}$ is generated by the quadratics \begin{equation} \label{weylgenerators} q_{a,b} = \sum_{i=1}^n x_{i,a} x_{i,b},\ \ \ \ \ \ 0\leq a\leq b.\end{equation} For $a>b$, define $q_{a,b} = q_{b,a}$, and let $\{Q_{a,b}|\ a,b\geq 0\}$ be commuting indeterminates satisfying $Q_{a,b} = Q_{b,a}$ and no other algebraic relations. The kernel $I_n$ of the homomorphism $\mathbb{C}[Q_{a,b}]\ra (\text{Sym} \bigoplus_{j\geq 0} V_j)^{O(n)}$ sending $Q_{a,b}\mapsto q_{a,b}$ is generated by the $(n+1)\times (n+1)$ determinants \begin{equation} \label{weylrel} d_{I,J} = \left[\begin{matrix} Q_{i_0,j_0} & \cdots & Q_{i_0,j_n} \cr  \vdots  & & \vdots  \cr  Q_{i_n,j_0}  & \cdots & Q_{i_n,j_n} \end{matrix} \right].\end{equation} Here $I=(i_0,\dots, i_{n})$ and $J = (j_0,\dots, j_{n})$ are lists of integers satisfying \begin{equation} \label{ijineq}0\leq i_0<\cdots <i_n,\ \ \ \ \ \ \ \ 0\leq j_0<\cdots <j_n.\end{equation} \end{thm}

The generators $q_{a,b}$ of $R$ correspond to vertex operators \begin{equation} \label{omegagen} \omega_{a,b} = \sum_{i=1}^n :\partial^a \alpha^i \partial^b \alpha^i: \ \in \cH(n)^{O(n)},\ \ \ \ \ \ \ \ 0\leq a\leq b\end{equation} satisfying $\phi_2(\omega_{a,b}) = q_{a,b}$. By Lemma \ref{reconlem}, $\{\omega_{a,b}|~0\leq a \leq b\}$ strongly generates $\cH(n)^{O(n)}$. A more economical choice of strong generating set for $\cH(n)^{O(n)}$ can be found by taking a minimal generating set for $R$ as a $\partial$-ring. The subset $\{q_{0,2m}|\ m\geq 0\}$ is easily seen to have this property, where $\partial q_{a,b} = q_{a+1,b} + q_{a,b+1}$. In fact, $\{q_{a,b}|\ 0\leq a\leq b\}$ and $\{\partial^k q_{0,2m}|\ k,m\geq 0\}$ both form bases for the degree two subspace of $R$. It follows from Lemma \ref{reconlem} that $\{\omega_{0,2m}|\ m\geq 0\}$ strongly generates $\cH(n)^{O(n)}$. As in \cite{LIII}, we use the notation \begin{equation} \label{defofj} j^{2m} = \omega_{0,2m},\ \ \ \ \ \ \ m\geq 0. \end{equation}

The modes $\{j^{2m}(k)| \ m\geq 0,\ k \in \mathbb{Z}\}$ generate a Lie algebra $\gB$ which was first studied by Bloch in \cite{Bl}. It is a central extension of a certain Lie subalgebra of the Lie algebra of differential operators on the circle. This Lie algebra also appeared in \cite{LIII}, although we were not aware of Bloch's work when \cite{LIII} was written. We constructed a $\gB$-module $\cV_n$ which has a vertex algebra structure, and is freely generated by fields $\{J^{2m}|\ m\geq 0\}$ satisfying the same OPE relations as the $j^{2m}$'s. We have a surjective homomorphism $$\pi_n: \cV_n\rightarrow \cH(n)^{O(n)},\ \ \ \ \ \ J^{2m}\mapsto j^{2m},$$ whose kernel $\cI_n$ is the maximal, proper graded $\gB$-submodule of $\cV_n$.

There is an alternative strong generating set $\{\Omega_{a,b}|~0\leq a\leq b\}$ for $\cV_n$ such that $\pi_n(\Omega_{a,b}) = \omega_{a,b}$. Recall the variables $Q_{a,b}$ and $q_{a,b}$ appearing in Theorem \ref{weylfft}. Since $\cV_n$ is freely generated by $\{J^{2m}|~m\geq 0\}$, and the sets $\{\Omega_{a,b}|~0\leq a\leq b\}$ and $\{\partial^k J^{2m}|~k,m\geq 0\}$ form bases for the same vector space, we may identify $\text{gr}(\cV_n)$ with  $\mathbb{C}[Q_{a,b}]$, and we identify $\text{gr}(\cH(n)^{O(n)})$ with $\mathbb{C}[q_{a,b}]/I_n$. Under this identification, $\text{gr}(\pi_{n}): \text{gr}(\cV_n) \ra \text{gr}(\cH(n)^{O(n)})$ is just the quotient map sending $Q_{a,b}\mapsto q_{a,b}$. There is a good increasing filtration on $\cV_{n}$ such that $(\cV_n)_{(2k)}$ is spanned by iterated Wick products of the generators $\Omega_{a,b}$, of length at most $k$, and $(\cV_n)_{(2k+1)} = (\cV_n)_{(2k)}$. Equipped with this filtration, $\cV_n$ lies in the category $\cR$, and $\pi_n$ is a morphism in $\cR$. 

\begin{lemma} \label{ddef} For each $I,J$ satisfying \eqref{ijineq}, there exists a unique element $D_{I,J}\in (\cV_{n})_{(2n+2)}$ of weight $|I| + |J| +2n+2$, where $|I| =\sum_{a=0}^{n} i_a$ and $|J| =\sum_{a=0}^{n} j_a$, satisfying 
\begin{equation} \label{uniquedij}\phi_{2n+2}(D_{I,J}) = d_{I,J},\ \ \ \ \ \ \pi_{n}(D_{I,J}) = 0.\end{equation}
These elements generate $\cI_{n}$ as a vertex algebra ideal. \end{lemma} 

\begin{proof} Clearly $\pi_{n}$ maps each filtered piece $(\cV_n)_{(k)}$ onto $(\cH(n)^{O(n)})_{(k)}$, so the hypotheses of Lemma \ref{idealrecon} are satisfied. Since $I_{n} = \text{Ker} (\text{gr}(\pi_{n}))$ is generated by the determinants $d_{I,J}$, we can apply Lemma \ref{idealrecon} to find $D_{I,J}\in (\cV_{n})_{(2n+2)} \cap \cI_{n}$ satisfying $\phi_{2n+2}(D_{I,J}) = d_{I,J}$, such that $\{D_{I,J}\}$ generates $\cI_{n}$ as a vertex algebra ideal. If $D'_{I,J}$ also satisfies \eqref{uniquedij}, we would have $D_{I,J} - D'_{I,J}\in (\cV_{n})_{(2n)} \cap \cI_{n}$. Since there are no relations in $\cH(n)^{O(n)}$ of degree less than $2n+2$, we have $D_{I,J} = D'_{I,J}$. \end{proof}

Let $D_{I,J}^{2n+2}\in (\cV_{n})_{(2n+2)}$ be some normal ordering of $d_{I,J}$, that is, a normally ordered polynomial obtained by replacing each $Q_{a,b}$ with $\Omega_{a,b}$, and replacing ordinary products with iterated Wick products. Then $\pi_{n}(D_{I,J}^{2n+2}) \in (\cH(n)^{O(n)})_{(2n)}$. Using the procedure outlined in Section \ref{secivt}, we can find a sequence of quantum corrections $D^{2k}_{I,J}$ for $k=1,\dots,n$, which are homogeneous, normally ordered polynomials of degree $k$ in the variables $\Omega_{a,b}$, such that $\sum_{k=1}^{n+1} D^{2k}_{I,J}$ lies in the kernel of $\pi_n$. We have 
\begin{equation}\label{decompofd} D_{I,J} = \sum_{k=1}^{n+1}D^{2k}_{I,J},\end{equation} since $D_{I,J}$ is uniquely characterized by (\ref{uniquedij}). 

In this decomposition, the term $D^2_{I,J}$ lies in the space $A_l$ spanned by $\{\Omega_{a,b}|~a+b=l\}$, for $l = |I| +|J|+2n$. From now on, we will restrict ourselves to the case where $l$ is an {\it even} integer $2m$. In this case, $A_{2m}  = \partial^2 (A_{2m-2}) \oplus \bra J^{2m}\ket$, where $\bra J^{2m}\ket$ denotes the linear span of $J^{2m}$, and we define $\text{pr}_{2m}: A_{2m}\ra \bra J^{2m}\ket$ to be the projection onto the second term. In \cite{LIII} we defined the {\it remainder} $R_{I,J}$ to be $\text{pr}_{2m}(D^2_{I,J})$. We write \begin{equation} \label{remcoeff} R_{I,J} = R_n(I,J) J^{2m},\end{equation} so that $R_n(I,J)$ denotes the coefficient of $J^{2m}$ in $\text{pr}_{2m}(D^2_{I,J})$. By Lemma 4.4 of \cite{LIII}, $R_n(I,J)$ is independent of the choice of decomposition (\ref{decompofd}). 

Suppose that $R_n(I,J) \neq 0$ for some $I,J$. Since $\pi_n(D_{I,J}) = 0$, we obtain a decoupling relation $$j^{2m} = P(j^0, j^2 \dots, j^{2m-2}),\ \ \ \ \ \ \ m = \frac{1}{2}(|I| + |J| + 2n),$$ where $P$ is a normally ordered polynomial in $j^0, j^2,\dots, j^{2m-2}$ and their derivatives. By applying the operator $j^2\circ_1$ repeatedly, we can construct higher decoupling relations $$j^{2r} = Q_{2r}(j^0,j^2,\dots,j^{2m-2})$$ for all $r > m$. The argument is the same as the proof of Theorem 4.7 of \cite{LIII}, and this implies that $\cH(n)^{O(n)}$ is strongly generated by $\{j^0, j^2, \dots, j^{2m-2}\}$. If $m$ is the {\it minimal} integer such that $R_n(I,J) \neq 0$ for some $I,J$ with $m = \frac{1}{2}(|I|+|J| + 2n)$, the strong generating set $\{j^0,j^2,\dots,j^{2m-2}\}$ for $\cH(n)^{O(n)}$ is minimal.

For $I = (0,1,\dots,n) = J$, $D_{I,J}$ is the unique element of the ideal $\cI_n$ (up to scalar multiples) of minimal weight $n^2+3n+2$. Define $R_n=R_n(I,J)$ in this case. In \cite{LIII}, we conjectured that $R_n \neq 0$, and we proved this for $n\leq 3$. This conjecture implies that $\{j^0, j^2, \dots, j^{n^2+3n-2}\}$ is a minimal strong generating set for $\cH(n)^{O(n)}$, and in particular $\cH(n)^{O(n)}$ is a $\cW$-algebra of type $\cW(2,4,\dots, n^2+3n)$.

\section{A recursive formula for $R_n(I,J)$}
In this section we will find a recursive formula for $R_n(I,J)$ in terms of $R_{n-1}(K,L)$ for lists $K,L$ of length $n$. Using this formula, we will show that $R_n \neq 0$ for $n\leq 6$. Although we do not prove that $R_n\neq 0$ in general, we will show that for each $n$, there exists {\it some} $I,J$ such that $R_n(I,J) \neq 0$. As above, this shows that $\cH(n)^{O(n)}$ is strongly finitely generated.

We assume throughout this section that $|I| + |J|+2n$ is even. Since $\text{pr}_{2m}(\Omega_{a,b}) =(-1)^{a} J^{2m}$ for $a+b = 2m$, it is easy to check that for $n=1$, $I = (i_0, i_1)$ and $J = (j_0, j_1)$, 
\begin{equation} \label{recurone} R_1(I,J) = \frac{(-1)^{i_0 + j_0}}{2 + i_0 + i_1}  - \frac{(-1)^{i_0 + j_1}}{2 + i_0 + i_1} + \frac{(-1)^{
 i_0 + i_1}}{2 + i_0 + j_0} - \frac{(-1)^{j_0 + j_1}}{2 + i_1 + j_0} \end{equation} $$  - \frac{ (-1)^{
 i_0 + i_1}}{2 + i_0 + j_1} + \frac{(-1)^{j_0 + j_1}}{2 + i_1 + j_1} - \frac{(-1)^{
 i_1 + j_0}}{2 + j_0 + j_1} + \frac{(-1)^{i_1 + j_1}}{2 + j_0 + j_1}.$$
 
Define a grading $\cH(n) = \bigoplus_{r\geq 0} \cH(n)^{(r)}$, where $\cH(n)^{(r)}$ is spanned by all normally ordered monomials $$:\partial^{k^1_1} \alpha^1 \cdots \partial^{k^1_{r_1}} \alpha^1\partial^{k^2_1} \alpha^2 \cdots \partial^{k^2_{r_2}} \alpha^2 \cdots \partial^{k^n_1} \alpha^n \cdots \partial^{k^n_{r_n}} \alpha^n:.$$ In this notation, $(k^i_1,\dots, k^i_{r_i})$ are sequences of integers satisfying $0\leq k^i_1 \leq k^i_2 \leq \cdots \leq k^i_{r_i}$ for all $i=1,\dots, n$, and $$\sum_{i=1}^n  k^i_1 + \cdots +k^i_{r_i} = r.$$ This grading is compatible with the good increasing filtration on $\cH(n)$ in the sense that $\cH(n)_{(r)} = \bigoplus_{j=0}^r \cH(n)^{(j)}$ and $\cH(n)^{(r)} \cong \cH(n)_{(r)} / \cH(n)_{(r-1)}$. Moreover, the action of $O(n)$ preserves the grading, so we obtain a grading $\cH(n)^{O(n)} = \bigoplus_{r\geq 0} (\cH(n)^{O(n)})^{(r)}$ and an isomorphism of graded vector spaces $$i_n: \cH(n)^{O(n)} \ra (\text{Sym} \bigoplus_{j\geq 0} V_j)^{O(n)},\ \ \ \ \ \ V_j \cong \mathbb{C}^n.$$ Let $p\in (\text{Sym}\bigoplus_{j\geq 0} V_j )^{O(n)}$ be a homogeneous polynomial of degree $2d$, and let $f = (i_{n})^{-1}(p)\in (\cH(n)^{O(n)})^{(2d)}$ be the corresponding homogeneous vertex operator. Let $F\in (\cV_{n})_{(2d)}$ be a vertex operator satisfying $\pi_{n}(F) = f$, where $\pi_{n}: \cV_{n}\rightarrow \cH(n)^{O(n)}$ is the projection. We can write $F = \sum_{k=1}^{d} F^{2k}$, where $F^{2k}$ is a normally ordered polynomial of degree $k$ in the vertex operators $\Omega_{a,b}$.

Let $\tilde{V}$ be the vector space $\mathbb{C}^{n+1}$, and let $$\tilde{q}_{a,b}\in (\text{Sym} \bigoplus_{j\geq 0} \tilde{V}_j)^{O(n+1)}$$ be the generator given by (\ref{weylgenerators}). Here $\tilde{V}_j$ is isomorphic to $\tilde{V}$ for $j\geq 0$. Let $\tilde{p}$ be the polynomial of degree $2d$ obtained from $p$ by replacing each $q_{a,b}$ with $\tilde{q}_{a,b}$, and let $\tilde{f} = (i_{n+1})^{-1} (\tilde{p}) \in (\cH(n+1)^{O(n+1)})^{(2d)}$ be the corresponding homogeneous vertex operator. Finally, let $\tilde{F}^{2k}\in \cV_{n+1}$ be the vertex operator obtained from $F^{2k}$ by replacing each $\Omega_{a,b}$ with the corresponding vertex operator $\tilde{\Omega}_{a,b}\in \cV_{n+1}$, and let $\tilde{F} = \sum_{k=1}^d \tilde{F}^{2k}$. 

\begin{lemma}  \label{corhomo} Fix $n\geq 1$, and let $D_{I,J}$ be an element of $\cV_n$ given by Lemma \ref{ddef}. There exists a decomposition $D_{I,J} = \sum_{k=1}^{n+1} D^{2k}_{I,J}$ of the form \eqref{decompofd} such that the corresponding vertex operator $$\tilde{D}_{I,J} = \sum_{k=1}^{n+1} \tilde{D}^{2k}_{I,J} \in \cV_{n+1}$$ has the property that $\pi_{n+1}(\tilde{D}_{I,J})$ lies in the homogeneous subspace $(\cH(n+1)^{O(n+1)})^{(2n+2)}$ of degree $2n+2$.
\end{lemma}

\begin{proof} The argument is the same as the proof of Corollary 4.14 of \cite{LI} and is omitted.
\end{proof}

Now we are ready to express $R_n(I,J)$ for all $I=(i_0,\dots, i_n)$, $J= (j_0,\dots, j_n)$ in terms of $R_{n-1}(K,L)$ for lists $K,L$ of length $n$. Let $d_{I,J}$ and $D_{I,J}$ be the corresponding elements of $\mathbb{C}[Q_{j,k}]$ and $\cV_{n}$, respectively. By expanding the determinant $d_{I,J}$ along its first column, we have $$d_{I,J} = \sum_{r=0}^n (-1)^r Q_{i_r,j_0} d_{I_r, J'},$$ where $I_r = (i_0,\dots, \widehat{i_r},\dots, i_n)$ is obtained from $I$ by omitting $i_r$, and $J' = (j_1,\dots,j_n)$ is obtained from $J$ by omitting $j_0$. Let $D_{I_r, J'} \in \cV_{n-1}$ be the vertex operator corresponding to the $n\times n$ determinant $d_{I_r,J'}$. By Lemma \ref{corhomo}, there exists a decomposition $$D_{I_r, J'} = \sum_{i=1}^n D^{2i}_{I_r, J'}$$ such that the corresponding element $\tilde{D}_{I_r, J'} = \sum_{i=1}^n \tilde{D}^{2i}_{I_r, J'} \in \cV_n$ has the property that $\pi_n(\tilde{D}_{I_r, J'})$ lies in the homogeneous subspace $(\cH(n)^{O(n)})^{(2n)}$ of degree $2n$. We have \begin{equation}\label{usefuli} \sum_{r=0}^n (-1)^r :\Omega_{i_r, j_0} \tilde{D}_{I_r,J'}:\  = \sum_{r=0}^n \sum_{i=1}^n (-1)^r :\Omega_{i_r, j_0} \tilde{D}^{2i}_{I_r,J'}:.\end{equation} The right hand side of \eqref{usefuli} consists of normally ordered monomials of degree at least $2$ in the vertex operators $\Omega_{a,b}$, and hence contributes nothing to the remainder $R_{I,J}$. Since $\pi_n(\tilde{D}_{I_r,J'})$ is homogeneous of degree $2n$, $\pi_n(:\Omega_{i_r, j_0} \tilde{D}_{I_r,J'}:)$ consists of a piece of degree $2n+2$ and a piece of degree $2n$ coming from all double contractions of $\Omega_{i_r, j_0}$ with terms in $\tilde{D}_{I_r,J'}$, which lower the degree by two. The component of $$\pi_n\bigg(\sum_{r=0}^n (-1)^r :\Omega_{i_r, j_0} \tilde{D}_{I_r,J'}:\bigg)\in \cH(n)^{O(n)}$$ in degree $2n+2$ must vanish since this sum corresponds to the classical determinant $d_{I,J}$. The component of $:\Omega_{i_r,j_0}\tilde{D}_{I_r,J'}:$ in degree $2n$ is $$S_r = (-1)^{i_r} \bigg(\sum_{k} \frac{\tilde{D}_{I_{r,k},J'}}{i_k +i_r+ 2}  + \sum_{l}  \frac{\tilde{D}_{I_r,J'_l}}{j_l + i_r+ 2} \bigg) + (-1)^{j_0} \bigg(\sum_{k} \frac{\tilde{D}_{I_{r,k},J'}}{i_k +j_0+ 2}  +\sum_{l}  \frac{\tilde{D}_{I_r,J'_l}}{j_l + j_0+ 2}  \bigg).$$ In this notation, for $k=0,\dots,n$ and $k\neq r$, $I_{r,k}$ is obtained from $I_r = (i_0,\dots, \widehat{i_r},\dots, i_n)$ by replacing the entry $i_k$ with $i_k+ i_r+j_0+2$. Similarly, for $l=1,\dots, n$, $J'_l$ is obtained from $J' = (j_1,\dots, j_n)$ by replacing $j_l$ with $j_l + i_r+j_0 + 2$. It follows that \begin{equation} \label{usefulii} \pi_n\bigg(\sum_{r=0}^n (-1)^r :\Omega_{i_r, j_0} \tilde{D}_{I_r,J'}:\bigg)= \pi_n\bigg(\sum_{r=0}^n (-1)^r S_r\bigg).\end{equation}

Combining \eqref{usefuli} and \eqref{usefulii}, we can regard $$\sum_{r=0}^n \sum_{i=1}^n (-1)^r :\Omega_{i_r, j_0} \tilde{D}^{2i}_{I_r,J'}: - \sum_{i=0}^n (-1)^r S_r$$ as a decomposition of $D_{I,J}$ of the form $D_{I,J} = \sum_{k=1}^{n+1} D^{2k}_{I,J}$ where the leading term $D^{2n+2}_{I,J} = \sum_{r=0}^n (-1)^r :\Omega_{i_r, j_0} \tilde{D}^{2n}_{I_r,J'}:$. It follows that $R_n(I,J)$ is the negative of the sum of the terms $R_{n-1}(K,L)$ corresponding to each $\tilde{D}_{K,L}$ appearing in $\sum_{r=0}^n (-1)^r S_r$. We therefore obtain the following recursive formula:
\begin{equation} \label{recursion}R_n(I,J) = -\sum_{r=0}^n (-1)^r(-1)^{i_r} \bigg(\sum_{k} \frac{R_{n-1}(I_{r,k},J')}{i_k +i_r+ 2}  + \sum_{l}  \frac{R_{n-1}(I_r,J'_l)}{j_l + i_r+ 2} \bigg) \end{equation} $$ -\sum_{r=0}^n (-1)^r (-1)^{j_0} \bigg(\sum_{k} \frac{R_{n-1}(I_{r,k},J')}{i_k +j_0+ 2}  +\sum_{l}  \frac{R_{n-1}(I_r,J'_l)}{j_l + j_0+ 2}  \bigg).$$

Recall that for $I = (0,1,\dots,n) = J$, $D_{I,J}$ is the unique element of $\cI_n$ (up to scalar multiples) of minimal weight $n^2+3n+2$, and $R_n=R_n(I,J)$ in this case. Using our recursive formula we can calculate the first few values of $R_n$:

\begin{table}[ht] 
\caption{Some values of $R_n$} 
\centering
\begin{tabular}{c c c c c c c} 
\hline\hline \\
n & 1 & 2 & 3 & 4 & 5 & 6 \\ [0.5ex]
\hline \\ $R_n$ & $\frac{5}{4} $ & $\frac{149}{600}$ & $-\frac{2419}{705600}$ & $-\frac{67619}{18670176000}$ & $\frac{1391081}{4879637199360000}$ &  $ \frac{40984649}{25145492674607585280000} $ \\ [1ex] \hline \end{tabular} \label{table:nonlin} \end{table}

In \cite{LIII}, the values $R_1$, $R_2$, and $R_3$ also appear, but $R_1$ was stated (erroneously) to be $-\frac{5}{4}$. Since $R_n\neq 0$ for $n\leq 6$, these calculations imply the following.
 
\begin{thm} For $n\leq 6$, $\cH(n)^{O(n)}$ has a minimal strong generating set $\{j^0,j^2,\dots,j^{n^2+3n-2}\}$, and in particular is a $\cW$-algebra of type $\cW(2,4,\dots, n^2+3n)$.
\end{thm}

Our next task is to show that for all $n\geq 1$, there exists $I,J$ such that $R_n(I,J) \neq 0$. Fix lists of non-negative integers $I = (i_0,i_1,\dots, i_{n-1})$ and $J = (j_0,j_1,\dots, j_{n-1})$ of length $n$. For each non-negative integer $x$, set $I_x = (i_0,i_1,\dots, i_{n-1},x)$ and $J_x = (j_0,j_1,\dots, j_{n-1},x)$. We may regard $R_n(I_x,J_x)$ as a function of $x$, and if we restrict $x$ to be even, $R_n(I_x,J_x)$ is a {\it rational} function of $x$. 

By applying the recursive formula \eqref{recursion} $k$ times, we can express $R_n(I_x, J_x)$ as a linear combination of terms of the form $R_{m}(A,B)$ where $m=n-k$, $A = (a_0,a_1,\dots, a_m)$, and $B = (b_0,b_1,\dots, b_m)$. Each entry of $A$ and $B$ is a constant plus a linear combination of entries from $I_x$ and $J_x$. Moreover, at most one entry of $A$ depends on $x$, and at most one entry of $B$ depends on $x$. Let $V_m$ denote the vector space spanned by elements $R_m(A,B)$ with these properties. Using \eqref{recursion}, we can express $R_m(A,B)$ as a linear combination of elements in $V_{m-1}$. A term $R_{m-1}(A_{r,k}, B')$ in this decomposition will be called {\it $x$-active} if either $a_r$, $a_k$, or $b_0$ depends on $x$. Similarly, $R_{m-1}(A_r, B'_l)$ will be called $x$-active if $a_r$, $b_0$, or $b_l$ depends on $x$. Define linear maps \begin{equation} f_m: V_m \rightarrow V_{m-1},\ \ \ \ \ \ \ g_m: V_m \rightarrow V_{m-1}\end{equation} as follows: $f_m(R_{m}(A,B))$ is the sum of the terms which are not $x$-active, and $g_m(R_m(A,B))$ is the sum of the terms which are $x$-active. Finally, define the {\it constant term map} \begin{equation} \label{conmap} c_m: V_m \rightarrow \mathbb{Q},\ \ \ \ \ \ \ c_m (R_m(A,B)) = \lim_{x\rightarrow \infty} R_m(A,B). \end{equation}

\begin{lemma} \label{indy} For all $n\geq 2$, $I = (i_0, i_1,\dots, i_{n-1})$ and $J = (j_0,j_1,\dots, j_{n-1})$ with $|I| + |J|$ even, we have \begin{equation} \label{constantid}c_n(R_n(I_x,J_x))= \frac{n}{|I|+|J|+2n} R_{n-1}(I,J).\end{equation} 
\end{lemma}

\begin{proof} For $n=2$ this is an easy calculation, so we may proceed by induction on $n$. In fact, it will be convenient to prove the following auxiliary formula at the same time:
\begin{equation} \label{constantidii} c_{n-1}(g_n(R_{n}(I_x, J_x))) = \frac{1}{|I|+|J|+2n} R_{n-1}(I,J).\end{equation} For $n=2$ this can be checked by direct calculation, so we assume both \eqref{constantid} and \eqref{constantidii} for $n-1$. Each term appearing in $f_n(R_{n}(I_x,J_x))$ is of the form $R_{n-1}(K,L)$ with $K = (k_0,k_1,\dots, k_{n-1},x)$ and $L = (l_0,l_1,\dots, l_{n-1},x)$, and $|K| + |L| + 2(n-1) = |I| + |J| + 2n$. By our inductive hypothesis that \eqref{constantid} holds for $n-1$, we have \begin{equation} \label{indyi}c_{n-1}(f_n(R_n(I_x, J_x))) = \frac{n-1}{|I|+|J|+2n} R_{n-1}(I,J).\end{equation}
Next, it is easy to check that $$c_{n-2}(g_{n-1}(f_n(R_{n}(I_x,J_x)))) = c_{n-2}(f_{n-1}(g_n(R_n(I_x,J_x)))).$$
Also, we have $$g_{n-1}(g_n(R_n(I_x, J_x))) = 0,$$ since all terms in this expression cancel pairwise. Therefore $$c_{n-1}(g_n(R_n(I_x, J_x))) = c_{n-2} (f_{n-1} (g_n(R_n(I_x, J_x)))) + c_{n-2} (g_{n-1} (g_n(R_n(I_x, J_x)))) $$ $$ = c_{n-2} (f_{n-1} (g_n(R_n(I_x, J_x)))) = c_{n-2} (g_{n-1} (f_n(R_n(I_x, J_x)))). $$ Moreover, by applying the induction hypothesis that \eqref{constantidii} holds for $n-1$, it follows that \begin{equation} \label{indyii}c_{n-1}(g_n(R_n(I_x, J_x))) = c_{n-2} (g_{n-1} (f_n(R_n(I_x, J_x)))) = \frac{1}{|I|+|J|+2n} R_{n-1}(I, J).\end{equation} Since $c_n(R_n(I_x,J_x)) = c_{n-1}(f_n(R_n(I_x,J_x)))+c_{n-1}(g_n(R_n(I_x,J_x)))$, the claim follows from \eqref{indyi} and \eqref{indyii}. \end{proof}

\begin{cor} For all $n\geq 1$, there exists $I,J$ such that $R_n(I,J) \neq 0$.
\end{cor}

\begin{proof} We have already shown this for $n\leq 6$, so we may assume inductively that we have found $I = (i_0,i_1,\dots , i_{n-1})$ and $J = (j_0, j_1, \dots, j_{n-1})$ such that $R_{n-1}(I,J) \neq 0$. By Lemma \ref{indy}, $c_n(R_n(I_x, J_x)) \neq 0$, so $R_n(I_x, J_x)$ defines a nontrivial rational function of $x$. This function has at most finitely many zeros, so we can find a positive even integer $x$ such that $R_n(I_x, J_x) \neq 0$. \end{proof}

\begin{thm} \label{heisresult} For all $n\geq 1$, $\cH(n)^{O(n)}$ is strongly generated by $\{j^0, j^2, \dots, j^{2m-2}\}$ for some $m$.
\end{thm}

\begin{proof} If $R_n(I,J) \neq 0$, we can take $m = \frac{1}{2}(|I|+|J| +2n)$.
\end{proof}

For an arbitrary reductive group $G$ of automorphisms of $\cH(n)$, the structure of $\cH(n)^G$ can be understood by decomposing $\cH(n)^{G}$ as a module over $\cH(n)^{O(n)}$. By Theorem 5.1 of \cite{LIII}, the Zhu algebra of $\cH(n)^{O(n)}$ is abelian, which implies that all its irreducible, admissible modules are highest-weight modules. Moreover, both $\cH(n)$ and $\cH(n)^G$ decompose as direct sums of irreducible, highest-weight $\cH(n)^{O(n)}$-modules. By Theorem 6.1 of \cite{LIII}, there is a finite set of irreducible $\cH(n)^{O(n)}$-submodules of $\cH(n)^G$ whose direct sum contains an (infinite) strong generating for $\cH(n)^G$. The key ingredient in the proof is Weyl's finiteness theorem (Theorem \ref{weylfinite}). This shows that $\cH(n)^G$ is finitely generated as a vertex algebra. If we assume the conjecture that $R_n\neq 0$, so that $\{j^0, j^2 ,\dots, j^{n^2+3n-2}\}$ strongly generates $\cH(n)^{O(n)}$, Theorem 6.9 of \cite{LIII} shows that $\cH(n)^G$ is strongly finitely generated. In fact, the only step where this conjecture was needed was the proof of Lemma 6.7 of \cite{LIII}, which gives a finiteness property of each irreducible, highest-weight $\cH(n)^{O(n)}$-submodule of $\cH(n)$. The proof of this lemma goes through if we replace $\frac{1}{2}(n^2+3n)$ with the integer $m$ appearing in Theorem \ref{heisresult} above. This proves the following.

\begin{thm} \label{heisgencase} For all $n\geq 1$ and any reductive group $G\subset O(n)$, $\cH(n)^G$ is strongly finitely generated.
\end{thm}

\section{Invariant subalgebras of affine vertex algebras}

Let $\gg$ be a finite-dimensional Lie algebra equipped with a nondegenerate, symmetric, invariant bilinear form $B$, and let $G$ be a reductive group of automorphisms of $V_k(\gg,B)$ preserving the conformal weight grading, for all $k\in \mathbb{C}$. In this section, we use a deformation argument together with Theorem \ref{heisgencase} to prove Theorem \ref{mainthm}, which states that $V_k(\gg,B)^G$ is strongly finitely generated for generic values of $k$. Fix an orthonormal basis $\{\xi_1,\dots, \xi_n\}$ for $\gg$ relative to $B$, so the generators $X^{\xi_i}\in V_k(\gg,B)$ satisfy $$X^{\xi_i}(z) X^{\xi_j}(w) \sim \delta_{i,j} k (z-w)^{-2} + X^{[\xi_i, \xi_j]}(w) (z-w)^{-1}.$$

Let $\kappa$ be a formal variable satisfying $\kappa^2 = k$, and let $F$ denote the $\mathbb{C}$-algebra of rational functions of the form $\frac{p(\kappa)}{\kappa^d}$ with $\text{deg}(p) \leq d$. Let $\cV$ be the vertex algebra with coefficients in $F$ freely generated by 
$\{a^{\xi_i} |\ i=1,\dots, n\}$, which satisfy the OPE relations $$a^{\xi_i}(z) a^{\xi_j}(w) \sim \delta_{i,j} (z-w)^{-2} + \frac{1}{\kappa}a^{[\xi_i, \xi_j]}(w) (z-w)^{-1}.$$ For $k\neq 0$, we have a surjective vertex algebra homomorphism $$\cV \ra V_k(\gg,B),\ \ \ \ \ a^{\xi_i} \mapsto \frac{1}{\sqrt{k}} X^{\xi_i},$$ whose kernel is the ideal $(\kappa - \sqrt{k})$, so $V_k(\gg,B) \cong \cV/ (\kappa - \sqrt{k})$. For each $m\geq 0$, the subspace $\cV[m]$ of weight $m$ is a free $F$-module, and 
$\text{rank}_{F} (\cV[m]) = \text{dim}_{\mathbb{C}} (V_k(\gg,B)[m])$ for all $k\in \mathbb{C}$. 

Since $F$ consists of rational functions of degree at most zero, the limit $$\cV_{\infty} = \lim_{\kappa \ra\infty} \cV$$ is a well-defined vertex algebra over $\mathbb{C}$ with generators $\alpha^{\xi_1},\dots, \alpha^{\xi_n}$ satisfying $$\alpha^{\xi_i}(z) \alpha^{\xi_j}(w) \sim \delta_{i,j} (z-w)^{-2}.$$ Therefore $\cV_{\infty}$ is isomorphic to $\cH(n)$. Define an $F$-linear map $\psi: \cV \ra \cV_{\infty}$ by \begin{equation} \label{clinmap} \psi\bigg(\sum_r c_r(\kappa) m_r(a^{\xi_i})\bigg) = \sum_{r} c_r m_r(\alpha^{\xi_i}),\ \ \ \ \ \ \ \ \ \ c_r = \lim_{\kappa \ra \infty} c_r(\kappa).\end{equation} In this notation, $m_r(a^{\xi_i})$ is a normally ordered monomial in $\partial^j a^{\xi_i}$, and $m_r(\alpha^{\xi_i})$ is obtained from $m_r(a^{\xi_i})$ by replacing each $a^{\xi_i}$ with $\alpha^{\xi_i}$. This map is easily seen to satisfy $\psi(\omega \circ_n \nu) = \psi(\omega) \circ_n \psi(\nu)$ for all $\omega,\nu \in \cV$ and $n\in \mathbb{Z}$.

Note that $\cV$ possesses a good increasing filtration, where $\cV_{(d)}$ is spanned by normally ordered monomials in $\partial^j a^{\xi_i}$ of degree at most $d$. We have an isomorphism of $F$-modules $\cV\cong \text{gr}(\cV)$, and isomorphisms of $\partial$-rings
$$\text{gr}(\cV) \cong F\otimes_{\mathbb{C}} (\text{Sym} \bigoplus_{j\geq 0} V_j) \cong F\otimes_{\mathbb{C}} \text{gr}(\cV_{\infty}).$$ Here $V_j \cong \gg$ for all $j$, and has a basis $\{\xi_{1,j},\dots, \xi_{n,j}\}$ corresponding to $\{\partial^{j} a^{\xi_1},\dots,\partial^j a^{\xi_n}\}\subset \cV$ and $\{\partial^{j} \alpha^{\xi_1},\dots,\partial^j \alpha^{\xi_n}\}\subset \cV_{\infty}$, respectively.

Since the group $G$ acts on $V_k(\gg,B)$ for all $k$, $G$ also acts on $\cV$ and preserves the formal variable $\kappa$. We have 
$$\text{gr}(\cV^G) \cong \text{gr}(\cV)^G \cong F\otimes_{\mathbb{C}} R \cong F\otimes_{\mathbb{C}} \text{gr}(\cV_{\infty})^G \cong F\otimes_{\mathbb{C}} \text{gr}((\cV_{\infty})^G),$$ where $R=(\text{Sym} \bigoplus_{j\geq 0} V_j)^G$. Finally, $\cV^G[m]$ is a free $F$-module and $$\text{rank}_F(\cV^G[m]) =\text{dim}_{\mathbb{C}} ((\cV_{\infty})^G[m]) = \text{dim}_{\mathbb{C}} (V_k(\gg,B)^G[m])$$ for all $m\geq 0$ and $k\in \mathbb{C}$.

The ring $R$ is graded by degree and weight, where $\xi_{1,j},\dots, \xi_{n,j}$ have degree $1$ and weight $j+1$. Choose a generating set $S = \{s_i|\ i\in I\}$ for $R$ as a $\partial$-ring, where $s_i$ is homogeneous of degree $d_i$ and weight $w_i$. We may assume that $S$ contains finitely many generators in each weight. By Lemma \ref{reconlem}, we can find a corresponding strong generating set $T = \{t_i|\ i\in I\}$ for $\cV^G$, where $$t_i\in (\cV^G)_{(d_i)},\ \ \ \ \ \ \ \ \phi_{d_i}(t_i) = 1\otimes s_i \in F\otimes_{\mathbb{C}} R.$$ Here $\phi_{d_i}: (\cV^G)_{(d_i)} \ra (\cV^G)_{(d_i)}/(\cV^G)_{(d_i-1)}\subset \text{gr}(\cV^G)$ is the usual projection. In particular, the leading term of $t_i$ is a sum of normally ordered monomials of degree $d_i$ in the variables $a^{\xi_1},\dots, a^{\xi_n}$ and their derivatives, and the coefficient of each such monomial is independent of $\kappa$. Let $u_i = \psi(t_i)$ where $\psi$ is given by \eqref{clinmap}, and define \begin{equation} (\cV^G)_{\infty} = \bra U\ket \subset (\cV_{\infty})^G, \end{equation} where $\bra U\ket$ is the vertex algebra generated by $\{u_i|\ i\in I\}$. 

Fix $m\geq 0$, and let $\{m_1,\dots, m_r\}$ be a set of normally ordered monomials in $t_i$ and their derivatives, which spans the subspace $\cV^G[m]$ of weight $m$. Then $(\cV^G)_{\infty}[m]$ is spanned by the corresponding monomials $\mu_l = \psi(m_l)$ for $l=1,\dots, r$, which are obtained from $m_l$ by replacing $t_i$ with $u_i$. Given normally ordered polynomials
$$P(u_i) = \sum_{l=1}^r c_l \mu_l \in  (\cV^G)_{\infty}[m],\ \ \ \ \ \ \ \ \tilde{P}(t_i) = \sum_{l=1}^r  c_l(\kappa) m_l \in\cV^G[m],$$ with $c_l \in \mathbb{C}$ and $c_l(\kappa) \in F$, we say that $\tilde{P}(t_i)$ {\it converges termwise} to $P(u_i)$ if $$\lim_{\kappa \ra \infty} c_l(\kappa) = c_l,\ \ \ \ \ \ l=1,\dots, r.$$ In particular, $\tilde{P}(t_i)$ converges termwise to zero if and only if $\lim_{\kappa \ra \infty} c_l(\kappa) = 0$ for $l=1,\dots, r$.

\begin{lemma} \label{keylemma} For each normally ordered polynomial relation $P(u_i)$ in $(\cV^G)_{\infty}$ of weight $m$ and leading degree $d$, there exists a relation $\tilde{P}(t_i) \in \cV^G$ of weight $m$ and leading degree $d$ which converges termwise to $P(u_i)$.
\end{lemma}

\begin{proof} We may write $P(u_i) = \sum_{a=1}^d P^a(u_i)$, where $P^a(u_i)$ is a sum of normally ordered monomials $\mu = :\partial^{j_1} u_{i_1} \cdots \partial^{j_t} u_{i_t}:$ of degree $a = d_{i_1} + \cdots +d_{i_t}$. The leading term $P^d(u_i)$ corresponds to a relation in $R$ among the generators $s_i$ and their derivatives, i.e., $P^d(s_i) = 0$. It follows that $P^d(t_i) \in (\cV^G)_{(d-1)}$. Since $P^a(u_i) \in ((\cV^G)_{\infty})_{(a)}$ for $a=1,\dots, d-1$, we have $P(t_i) \in (\cV^G)_{(d-1)}$. Since $\{t_i|\ i\in I\}$ strongly generates $\cV^G$, we can express $P(t_i)$ as a normally ordered polynomial $P_0(t_i)$ of degree at most $d-1$. Let $Q(t_i) = P(t_i) - P_0(t_i)$, which is therefore a relation in $\cV^G$ with leading term $P^d(t_i)$. 

If $P_0(t_i)$ converges termwise to zero, we can take $\tilde{P}(t_i) = Q(t_i)$ since $P(t_i)$ converges termwise to $P(u_i)$. Otherwise, $P_0(t_i)$ converges termwise to a nontrivial relation $P_1(u_i)$ in $(\cV^G)_{\infty}$ of degree at most $d-1$. By induction on the degree, we can find a relation $\tilde{P}_1(t_i)$ of leading degree at most $d-1$, which converges termwise to $P_1(u_i)$. Finally, $\tilde{P}(t_i) = P(t_i) - P_0(t_i) - \tilde{P}_1(t_i)$ is easily seen to have the desired properties. \end{proof}

\begin{cor} \label{ginfcom} $(\cV^G)_{\infty} = (\cV_{\infty})^G = \cH(n)^G$.
\end{cor}

\begin{proof} Recall that $\text{rank}_{F} (\cV^G[m]) = \text{dim}_{\mathbb{C}}((\cV_{\infty})^G[m])$ for all $m\geq 0$. Since $(\cV^G)_{\infty} \subset (\cV_{\infty})^G$, it suffices to show that  $\text{rank}_{F} (\cV^G[m]) = \text{dim}_{\mathbb{C}}((\cV^G)_{\infty}[m])$ for all $m\geq 0$. Let $\{m_1,\dots,m_r\}$ be a basis for $\cV^G[m]$ as an $F$-module, consisting of normally ordered monomials in $t_i$ and their derivatives. The corresponding elements $\mu_l = \psi(m_l)$ for $l=1,\dots, r$ span $(\cV^G)_{\infty}[m]$, and Lemma \ref{keylemma} implies that they are linearly independent as well. Otherwise, a nontrivial relation among $\mu_1,\dots, \mu_r$ would give rise to a nontrivial relation among $m_1,\dots, m_r$.
\end{proof}

By Theorem \ref{heisgencase} and Corollary \ref{ginfcom}, $(\cV^G)_{\infty}$ is strongly finitely generated. Without loss of generality, we may assume that $U = \{u_i| \ i\in I\}$ contains a finite subset $$U' = \{u_1,\dots, u_s\},$$ which strongly generates $(\cV^G)_{\infty}$. This implies that for each $u\in U \setminus U'$, there exists a decoupling relation $u = P(u_1\dots, u_s)$ where $P$ is a normally ordered polynomial in $u_1,\dots, u_s$ and their derivatives.

For each subset $K \subset \mathbb{C}$ containing $0$, let $F_K$ denote the $\mathbb{C}$-algebra of rational functions in $\kappa$ of the form $\frac{p(\kappa)}{q(\kappa)}$ where $\text{deg}(p)\leq \text{deg}(q)$ and all roots of $q$ lie in $K$. In this notation, $F = F_K$ for $K = \{0\}$. Given an element $\omega \in \cV^G$, we will denote the element $1\otimes \omega \in F_K\otimes_F \cV^G$ by $\omega$, when no confusion can arise.

\begin{thm} There exists a subset $K\subset \mathbb{C}$ containing $0$ which is at most countable, such that $F_K\otimes_F \cV^G$ is strongly generated by the subset $T' = \{t_1,\dots, t_s\}\subset T$ corresponding to $U'$.
\end{thm}

\begin{proof}
Let $d$ be the first weight such that $U$ contains elements which do not lie in $U'$, and let $u_{1,d},\dots, u_{r,d}$ be the set of elements of $U\setminus U'$ of weight $d$. Since $U'$ strongly generates $(\cV^G)_{\infty}$, we have decoupling relations in $(\cV^G)_{\infty}$ of the form $$u_{j,d} = P_j(u_1,\dots, u_s),\ \ \ \ \ j=1,\dots, r.$$ Let $t_{j,d}$ be the corresponding elements of $T$. By Lemma \ref{keylemma}, there exist relations $$t_{j,d} = \tilde{P}_j(t_1,\dots,t_s, t_{1,d},\dots, \widehat{t_{j,d}},\dots, t_{r,d}), \ \ \ \ \ \ \ \ j=1,\dots, r,$$ 
which converge termwise to $P_j(u_1,\dots, u_s)$. Here $\tilde{P}_j$ does not depend on $t_{j,d}$ but may depend on $t_{k,d}$ for $k\neq j$. Since each $t_{k,d}$ has weight $d$ and $\tilde{P}_j$ is homogeneous of weight $d$, $\tilde{P}_j$ depends linearly on $t_{k,d}$. We can therefore rewrite these relations in the form 
$$\sum_{k=1}^r b_{jk} t_{k,d} = Q_j(t_1,\dots, t_s),\ \ \ \ \ \ \ \ b_{jk} \in F,$$ where $b_{jj} = 1$, $\lim_{\kappa\ra \infty} b_{jk} = 0$ for $j\neq k$, and 
$$Q_j(t_1,\dots, t_s) = \tilde{P}_j(t_1,\dots,t_s, t_{1,d},\dots, \widehat{t_{j,d}},\dots, t_{r,d})  + \sum_{k=1}^{j-1} b_{jk} t_{k,d}  +\sum_{k=j+1}^r b_{jk} t_{k,d}.$$
Clearly $\lim_{\kappa \ra \infty} \det[b_{jk}] = 1$, so this matrix is invertible over the field of rational functions in $\kappa$. Let $K_d$ denote the set of distinct roots of the numerator of $\det[b_{jk}]$ regarded as a rational function of $\kappa$. We can solve this linear system over the ring $F_{K_d}$, so in $F_{K_d}\otimes_F \cV^G$ we obtain decoupling relations $$t_{j,d} = \tilde{Q}_j(t_1,\dots, t_s),\ \ \ \ \ \ \ \ j=1,\dots, r.$$ For each weight $d+1,d+2\dots$ we repeat this procedure, obtaining finite sets $$K_d\subset K_{d+1} \subset K_{d+2} \subset \cdots$$ and decoupling relations $$t = P(t_1,\dots, t_s)$$ in $F_{K_{d+i}}\otimes_F \cV^G$, for each $t \in T\setminus T'$ of weight $d+i$. Letting $K  = \bigcup_{i\geq 0} K_{d+i}$, we obtain a complete set of decoupling relations in $F_K \otimes_F \cV^G$ expressing each $t \in T \setminus T'$ as a normally ordered polynomial in $t_1,\dots,t_s$ and their derivatives. \end{proof}

\begin{proof}[Proof of Theorem \ref{mainthm}] 
Recall that for $k\neq 0$ we have $\cV^G / (\kappa - \sqrt{k}) \cong V_k(\gg,B)^G$. For all $k$ such that $\sqrt{k} \notin K$, we have $$\cV^G / (\kappa - \sqrt{k}) \cong (F_K \otimes_F \cV^G)/ (\kappa - \sqrt{k}) \cong (F_K \otimes_F \bra T' \ket )/ (\kappa - \sqrt{k}) \cong V_k(\gg,B)^G,$$ where $\bra T' \ket$ denotes the space of normally ordered polynomials in $t_1,\dots, t_s$ and their derivatives. It follows that when $\sqrt{k}\notin K$, $V_k(\gg,B)^G = \bra T'_k \ket$, where $T'_k$ is the image of $T'$ under the map $\cV^G \ra V_k(\gg,B)^G$. In particular, $V_k(\gg,B)^G$ is strongly generated by $T'_k$. \end{proof}

\begin{proof}[Proof of Corollary \ref{maincor}] Since $F_{K} \otimes_F \bra T' \ket = F_K \otimes_F \cV^G$, we may define the structure constants of the pair $(F_{K} \otimes_F \cV^G,T')$ to be the coefficients of each normally ordered monomial in the generators and their derivatives appearing in $t_i(z) \circ_n t_j(w)$ for $i,j=1,\dots,s$ and $n\geq 0$. Let $D$ be the set of values of $\kappa$ for which all structure constants are defined. Since there are only finitely many structure constants, and each structure constant is a rational function of $\kappa$, we have $D = \mathbb{C} \setminus K'$ for some finite set $K'\subset K$. 


For all $k$ such that $\sqrt{k} \in D$, define $\cW(\gg,B,G)_k$ to be the subalgebra of $V_k(\gg,B)^G$ generated by $T'_k$. Clearly $\cW(\gg,B,G)_k$ is {\it strongly} generated by $T'_k$. When $\sqrt{k} \notin K$ we have $\cW(\gg,B,G)_k = V_k(\gg,B)^G$, although $\cW(\gg,B,G)_k$ may be a proper subalgebra of $V_k(\gg,B)^G$ if $\sqrt{k} \in D \cap K$. Since the structure constants in $\cW(\gg,B,G)_k$ are rational functions of $\sqrt{k}$, it is immediate that $\cW(\gg,B,G)_k$ is a deformable $\cW$-algebra with the desired properties. \end{proof}

\section{A concrete example: $\gg = \gs\gl_2$ and $G = SL(2)$} In this section we consider the $SL(2)$-invariant subalgebra of $V_k(\gs\gl_2)$ in order to give the reader a feeling for the structure of the invariant vertex algebras studied in this paper. The vertex algebra $V_k(\gs\gl_2)^{SL(2)}$ was previously studied in \cite{BFH}, where it was conjectured to be strongly finitely generated. We work in the usual root basis $x,y,h$ satisfying $$[x,y]=h,\ \ \ \ \ \  [h,x]=2x,\ \ \ \ \ \ [h,y]=-2y.$$ The generators $X^x, X^y, X^h$ of $V_k(\gs\gl_2)$ satisfy
$$X^x(z) X^y(w) \sim k (z-w)^{-2} + X^h(w) (z-w)^{-1},\ \ \ \ \ \ \ \ \ \ X^h(z) X^h(w) \sim  2k (z-w)^{-2},$$
$$X^h(z) X^x(w) \sim  2 X^x(w) (z-w)^{-1},\ \ \ \ \ \ \ \ \ \ X^h(z) X^y(w) \sim  -2 X^y (w)(z-w)^{-1}.$$ Recall that $\gs\gl_2 \cong \gs\go_3$ as complex Lie algebras, and the adjoint representation $V$ of $\gs\gl_2$ is isomorphic to the standard representation $W$ of $\gs\go_3$. Hence for $V_j \cong V$ and $W_j\cong W$, we may identify $(\text{Sym} \bigoplus_{j\geq 0} V_j)^{SL(2)}$ and $(\text{Sym} \bigoplus_{j\geq 0} W_j)^{SO(3)}$. The following theorem is due to Weyl; see Theorem 2.9A and Section 2.17 of \cite{W}.

\begin{thm}\label{weyl} For $n\geq 0$, let $V_n$ be a copy of the adjoint representation of $\gs\gl_2$, with basis $\{a^h_n,a^x_n,a^y_n\}$. Then $(\text{Sym} \bigoplus_{n=0}^{\infty}V_n)^{SL(2)}$ is generated by \begin{equation}\label{quadgen} q_{ij}= a^h_i a^h_j + 2a^x_i a^y_j + 2a^x_j a^y_i,\ \ \ \ \ \ \ i,j\geq 0, \end{equation}
\begin{equation}\label{cubgen}c_{klm}= \left| \begin{array}{lll} a^h_k & a^x_k & a^y_k \\ a^h_l & a^x_l & a^y_l \\ a^h_m & a^x_m & a^y_m \end{array}\right|,\ \ \ \ \ \ \ 0\leq k<l<m. \end{equation} The ideal of relations among the variables $q_{ij}$ and $c_{klm}$ is generated by polynomials of the following two types: \begin{equation}\label{firstrel} q_{ij}c_{klm}-q_{kj}c_{ilm}+q_{lj}c_{kim}-q_{mj}c_{kli}, \ \ \ \ \ \ \ \ \ \ c_{ijk}c_{lmn}+\frac{1}{4}\left| \begin{array}{lll} q_{il} & q_{im} & q_{in} \\q_{jl} & q_{jm} & q_{jn} \\ q_{kl} & q_{km} & q_{kn}\end{array}\right| . \end{equation} \end{thm}

We have linear isomorphisms \begin{equation} \label{solin} \cH(3)^{SO(3)}\cong \text{gr}(\cH(3))^{SO(3)} \cong (\text{Sym} \bigoplus_{j\geq 0} V_j)^{SL(2)}\cong \text{gr}(V_k(\gs\gl_2))^{SL(2)} \cong V_k(\gs\gl_2)^{SL(2)}\end{equation} and isomorphisms of $\partial$-rings \begin{equation} \label{sodr} \text{gr}(\cH(3)^{SO(3)})\cong (\text{Sym} \bigoplus_{j\geq 0} V_j)^{SL(2)}\cong \text{gr}(V_k(\gs\gl_2)^{SL(2)}).\end{equation} The generating set $\{q_{ij},c_{klm}\}$ for $(\text{Sym} \bigoplus_{j\geq 0} V_j)^{SL(2)}$ corresponds to strong generating sets $\{Q_{ij}, C_{klm}\}$ for $\cH(3)^{SO(3)}$ and $\{\tilde{Q}_{ij}, \tilde{C}_{klm}\}$ for $V_k(\gs\gl_2)^{SL(2)}$, respectively. In this notation, $Q_{ij}$ corresponds to the element $\omega_{i,j}$ given by (\ref{omegagen}), and $Q_{0,2m}$ corresponds to $j^{2m}$. Both $Q_{ij}$ and $\tilde{Q}_{ij}$ have weight $i+j+2$, and $C_{klm}$ and $\tilde{C}_{klm}$ have weight $k+l+m+3$. In terms of the generators $X^x, X^y, X^h$ for $V_k(\gs\gl_2)$, 
\begin{equation} \label{deftildeq} \tilde{Q}_{ij} = :\partial^i X^h \partial^jX^h: + 2 :\partial^i X^x \partial^j X^y: + 2 :\partial^i X^y \partial^j X^x:,\end{equation}
$$ \tilde{C}_{klm} = : \partial^k X^x \partial^l X^y \partial^m X^h:  - :\partial^k X^x \partial^m X^y \partial^l X^h:  -  :\partial^l X^x \partial^k X^y  \partial^m X^h:  + : \partial^l X^x \partial^m X^y \partial^k X^h: $$ \begin{equation} \label{deftildec} + :\partial^m X^x  \partial^k X^y  \partial^l X^h:  -  :\partial^m X^x \partial^l X^y \partial^k X^h:.\end{equation} 


Next, we find a strong finite generating set for $\cH(3)^{SO(3)}$, using the structure of $\cH(3)^{SO(3)}$ as a module over $\cH(3)^{O(3)}$ and the fact that $\cH(3)^{O(3)}$ is of type $\cW(2,4,\dots, 18)$. By Theorem 6.1 of \cite{LIII}, for any $G \subset O(3)$, there is a finite set of irreducible, highest-weight $\cH(3)^{O(3)}$-submodules of $\cH(3)$ whose direct sum contains a strong generating set for $\cH(3)^G$. In the case $G=SO(3)$, we only need two such modules $\cM_0$ and $\cM$, where $\cM_0 =\cH(3)^{O(3)}$, which has highest-weight vector $1$ and contains all the quadratics $Q_{ij}$, and $\cM$ has highest-weight vector $C_{012}$ and contains all the cubics $C_{klm}$.

The Lie algebra $\cP$ generated by $\{j^{2l}(k)|\ k,l\geq 0\}$ acts by derivations of degree zero on $\text{gr}(\cH(3))^{SO(3)}$. Let $\cM' \subset \cM$ be the $\cP$-submodule of $\cM$ generated by $C_{012}$, which lies in the filtered component $\cM_{(3)}$ since $C_{012}\in\cM_{(3)}$ and $\cP$ preserves the filtration. By Lemma 6.6 of \cite{LIII}, $\cM'$ is spanned by elements of the form \begin{equation}\label{genmprime}j^{2l_1}(k_1)j^{2l_2}(k_2)j^{2l_3}(k_3) C_{012},\ \ \ \ \ \ \ k_i \leq 5,\end{equation} since we have $d = 3$ and $m=2$ in this case. Without loss of generality, we may assume that $l_i \geq l_j$ for $i,j = 1,2,3$, and $k_i \leq k_j$ whenever $l_i = l_j$. Letting $\cH(3)^{O(3)}[k]$ denote the subspace of weight $k$, define the {\it Wick ideal} $\cM_{Wick} \subset \cM$ to be the subspace spanned by $$:a(z) b(z):, \ \ \ a(z) \in \bigoplus_{k>0} \cH(3)^{O(3)}[k],\ \ \ \ b(z)\in \cM.$$ By Lemma 6.7 of \cite{LIII}, $\cM/ \cM_{Wick}$ is finite-dimensional, so there exists a finite set $S\subset \cM$ such that $\cM$ is spanned by the elements $$:\omega_1(z) \cdots \omega_t(z) \alpha(z):,\ \ \ \ \ \omega_j(z) \in \cH(3)^{O(3)},\ \ \ \ \ \alpha(z) \in S.$$ 

The proof of Lemma 6.7 of \cite{LIII} shows that $S$ can be taken to be a subset of $\cM'$. In fact, we claim that $S$ can be taken to be any basis for the subspace of $\cM'$ of weight at most $327$. Since $j^{2l}(k)$ has weight $2l-k+1$ and $C_{012}$ has weight $6$, $(j^{106}(0) )^3 C_{012}$ has weight $327$. It follows that for any element $\omega$ of the form (\ref{genmprime}) of weight greater than $327$, we must have $2l_1 \geq 108$. The minimal weight of $j^{2l_1}(k_1)$ is $104$ and occurs if $2l_1 = 108$ and $k_1 = 5$. Using the decoupling relation $j^{2l_1} = Q_{2l_1}(j^0,\dots, j^{16})$, the operator $j^{2l_1}(k_1)$ can be expressed as a linear combination of operators of the form \begin{equation} \label{rewrite}j^{2r_1}(s_1) \cdots j^{2r_t}(s_t),\ \ \ \ \ 2r_i \leq 16,\ \ \ \ \ s_i \in \mathbb{Z} .\end{equation} Using the formula \begin{equation}\label{vaidiv} (:ab:)\circ_n
c=\sum_{k\geq0}{1\over k!}:(\partial^ka)(b\circ_{n+k}c):
+(-1)^{|a||b|}\sum_{k\geq0}b\circ_{n-k-1}(a\circ_k c),\end{equation} which holds for any $n>0$ and any vertex operators $a,b,c$ in a vertex algebra $\cA$, it is easy to check that $j^{2l_1}(k_1)$ is a linear combination of terms of the form (\ref{rewrite}) with $1\leq t\leq 6$, and $\sum_{i=1}^t  s_i = 6-t$. The maximal weight of such terms involving no creation operators is $102$, since $(j^{16}(0))^6$ has weight $102$. Since $j^{2l_1}(k_1)$ has weight at least $104$, each term of the form (\ref{rewrite}) will involve creation operators. Therefore $\omega$ lies in the Wick ideal. 

Since all terms of the form (\ref{genmprime}) are linear combinations of the $C_{klm}$, we may take $$S = \{C_{klm}| \  \text{wt}(C_{klm}) = k+l+m +3 \leq 327\}.$$ Theorem 6.9 of \cite{LIII} then shows that $S \cup \{Q_{0,2j}| \  0\leq j \leq 8\}$ is a strong finite generating set for $\cH(3)^{SO(3)}$. Finally, we obtain

\begin{thm} The corresponding set $$\{\tilde{C}_{klm}| \ k+l+m \leq 324\} \cup \{\tilde{Q}_{0,2j}|\ 0\leq j \leq 8\} \subset V_k(\gs\gl_2)^{SL(2)},$$ where $\tilde{Q}_{ij}$ and $\tilde{C}_{klm}$ are given by (\ref{deftildeq})-(\ref{deftildec}), is a strong finite generating set for $V_k(\gs\gl_2)^{SL(2)}$ for generic values of $k$. \end{thm}

Computer calculations indicate that this strong generating set is larger than necessary. We hope to return to the question of finding minimal strong generating sets for these vertex algebras in the future.

\end{document}